\documentclass[a4paper,12pt]{amsart}
\usepackage[T1]{fontenc}
\usepackage[utf8]{inputenc}
\usepackage[english]{babel}
\usepackage{amsmath, amssymb, amsthm, amsfonts}
\usepackage{mathrsfs}
\usepackage{tikz}
\usetikzlibrary{arrows}
\usetikzlibrary{matrix}
\usepackage[left=2.5cm, right=2.5cm, top=2cm]{geometry}
\usepackage{enumerate}
\usepackage{hyperref}
\date{}
\theoremstyle{plain}
\newtheorem{thm}{Theorem}[section]
\newtheorem{lem}{Lemma}[section]

\newtheorem{prop}{Proposition}[section]
\newtheorem{algo}{Algorithm}[section]
\nonstopmode\numberwithin{equation}{section}
\theoremstyle{definition}

\newtheorem{rema}{Remark}[section]

\newtheorem*{ques*}{Question}

\makeatletter
\def\tagform@#1{\maketag@@@{\ignorespaces#1\unskip\@@italiccorr}}
\let\orgtheequation\theequation
\def\theequation{(\orgtheequation)}
\makeatother
\let\orgautoref\autoref

\renewcommand{\autoref}[1]{\def\equationautorefname{}\orgautoref{#1}}

\usepackage{fancyhdr}

\newcommand\shorttitle{The Short Title}

\begin{document}

\title[A novel  two-point gradient
method]{A novel two-point gradient method for Regularization of inverse problems in Banach spaces}
\author{ Gaurav Mittal, Ankik Kumar Giri}
\email{gmittal@ma.iitr.ac.in, ankik.giri@ma.iitr.ac.in}
\address{Department of Mathematics, Indian Institute of Technology Roorkee, Roorkee, India}
\maketitle
\begin{abstract}
In this paper, we introduce a novel two-point gradient method for solving the ill-posed problems in Banach spaces and study its convergence analysis. The method is based on the well known iteratively regularized  Landweber iteration method together with an extrapolation strategy. The general formulation of iteratively regularized  Landweber iteration method in Banach spaces excludes the use of certain functions such as total variation like penalty functionals, $L^1$ functions etc.  The novel scheme presented in this paper allows to use such non-smooth penalty terms   that can be helpful in practical applications involving the reconstruction of  several important features of solutions such as piecewise constancy and sparsity. We carefully  discuss the choices for important parameters, such as  combination parameters and step sizes involved in the design of the method. Additionally, we discuss an example to validate our assumptions.
\end{abstract}
\begin{center}

\hspace{-12mm}\textbf{Keywords:}   Regularization,  Iterative methods, Two point gradient method\\ \smallskip
\hspace{-50mm}\subjclass {}{\textbf{AMS Subject Classifications}: 65J15, 65J20, 47H17}\\
\end{center}\medskip\section{Introduction}
Let $F:D(F)\subset U\to V$ be an operator between the Banach spaces $U$ and $V$, with domain $D(F)$. In this paper, our main aim is to solve the following inverse problems \begin{equation}\label{1.1}
F(u)=v.
\end{equation}
In general, due to unstable dependence of
solutions on the small  data perturbations,  inverse problems of the form \autoref{1.1} are ill-posed in nature. Throughout this paper, we assume that the data in $(1.1)$ is attainable, i.e. \autoref{1.1} has a solution, which may not be unique. Instead of the exact data $v$, we assume the availability of perturbed data $v^{\delta}$ satisfying \begin{equation}\label{1.2}
\|v-v^{\delta}\|\leq \delta.
\end{equation}
Consequently, in order to obtain the approximate solutions of \autoref{1.1}, regularization methods  are required. In Hilbert spaces, one of the  most prominent regularization method is  Landweber iteration due to its simplicity and  robustness
with respect to noise. We refer to \cite{Engl, Hanke} for the detailed study of this method in linear and  nonlinear inverse problems.

Due to the tendency of classical Landweber iteration to over-smooth the solutions  in Hilbert spaces,   it is difficult to deduce special features of the desired solution such as discontinuity and sparsity through this scheme. Therefore, various
modifications of Landweber iteration have been proposed in Banach spaces  to overcome this drawback, see \cite{Bot, Jin1, Jin2, Kalten, Schuster, Wang} etc.  In \cite{Schuster, Kalten}, the following Landweber iteration scheme has been proposed:\begin{equation} \label{(1.3)}
\begin{split}\Im_{k+1}^{\delta} =\Im_{k}^{\delta}-\upsilon_k^{\delta}F'(u_k^{\delta})^*J_s^V((F(u_k^{\delta})-v^{\delta}),\\ \newline u_{k+1}^{\delta} =J_q^{U^*}(\Im_{k+1}^{\delta}),\hspace{37mm}
\end{split}
\end{equation}
where $F'(u)^*$ denote the adjoint of  Fr\'echet derivative $F'(u)$ of $F$ at $u$, $\upsilon_k^{\delta}$ is the step size, $J_s^V:V\to V^*$ and $J_q^{U^*}:U^*\to U$ are the duality mappings with gauge functions $x\to x^{s-1}$ and $x\to x^{q-1}$, respectively with $1<s, q <\infty$.
Basically, Landweber iteration \autoref{(1.3)} has been obtained by applying a gradient method for solving the minimization problem $\min \frac{1}{s}\|F(u)-v^{\delta}\|^s.$ Motivated by the Landweber iteration \autoref{(1.3)}, the  following modification of Landweber iteration well known as  iteratively regularized Landweber iteration  has been given in \cite{Schuster, Kalten}:\begin{equation} \label{(1.4)}
\begin{split}\Im_{k+1}^{\delta} =(1-\alpha_k)\Im_{k}^{\delta}-\upsilon_k^{\delta}F'(u_k^{\delta})^*J_s^V((F(u_k^{\delta})-v^{\delta})+\alpha_k\Im_0,\\ \newline u_{k+1}^{\delta} =J_q^{U^*}(\Im_{k+1}^{\delta}),\hspace{67mm}
\end{split}
\end{equation}
where $\{\alpha_k\}\in [0, 1]$  is an appropriately chosen sequence, $\Im_0=\Im_0^{\delta}$,  and $u_0\in U$ is an initial point. The additional term $\alpha_k(\Im_0-\Im_k^{\delta})$ in the method \autoref{(1.4)} compared to method \autoref{(1.3)} is motivated by the well known iteratively regularized Gauss-Newton method (cf. \cite{Scherzer}). The respective formulations \autoref{(1.3)}, \autoref{(1.4)} of Landweber and iteratively regularized Landweber iterations, however, are not defined for incorporating the $L^1$ and the total variation like penalty functionals. With general uniformly convex penalty functionals, Landweber-type iteration was introduced in \cite{Bot, Jin2} for linear as well as non-linear ill-posed problems. The method formulated in \cite{Bot, Jin2} can be written as \begin{equation}\label{(1.5)} 
\begin{split}\Im_{k+1}^{\delta} =\Im_{k}^{\delta}-\upsilon_k^{\delta}F'(u_k^{\delta})^*J_s^V((F(u_k^{\delta})-v^{\delta}),\\ \newline u_{k+1}^{\delta} =\arg \min_{u\in U}\big\{\varphi(u)-\langle \Im_{k+1}^{\delta}, u\rangle\big\},\hspace{8mm}
\end{split}
\end{equation}
where $\varphi:U\to(-\infty, \infty)$ is a proper uniformly convex semi-continuous functional. The advantage of this method is that the functional $\varphi$ can be wisely chosen so that it can be utilized in determining different characteristics of the solution.

Despite  the simplicity in the implementation of  Landweber iteration, various other newton type methods have been investigated in the literature, primarily due to its slowness \cite{Kalten}. The newton type methods are comparatively faster than Landweber iteration,  however, while dealing with each iteration step they always  spend more computational time. Therefore, as desired,  by preserving the implementation simplicity of Landweber iteration, various  accelerated Landweber iterations have been proposed in the literature.

Based on orthogonal polynomials and  spectral theory,   Hanke \cite{Hanke2} proposed  a family of accelerated Landweber iterations for linear inverse problems in Hilbert spaces. But this accelerated family  is no longer available to use general convex penalty functionals. Hein  et al. \cite{Hein} presented an  accelerated Landweber iteration in Banach spaces by carefully choosing the step size of each iteration. After then, using the  sequential subspace optimization strategy, different versions of accelerated Landweber iteration have been discussed in \cite{Hegland, Shop}. Recently, the following accelerated  Landweber iteration based on Nesterov’s strategy  \cite{Nest} (this strategy was originally proposed  to accelerate the gradient method) has been discussed (cf. \cite{Jin3}): \begin{equation}\label{(1.6)}\begin{split}
w_k^{\delta}=u_k^{\delta}+\frac{k}{k+\varsigma}(u_k^{\delta}-u_{k-1}^{\delta}),\\ 
u_{k+1}^{\delta}=w_k^{\delta}-\upsilon_k^{\delta}F'(u_k^{\delta})^*((F(u_k^{\delta})-v^{\delta}),\hspace{-11mm}
\end{split}
\end{equation}
where $\varsigma\geq 3$, $u_{-1}^{\delta}=u_0^{\delta}=u_0$ is an initial guess. A further modification of \autoref{(1.6)} known as two-point gradient method was proposed in \cite{Hub} by substituting general connection parameters $\lambda_k^{\delta}$ in place of $\frac{k}{k+\varsigma}$. Very recently, Zhong et al. \cite{Zhong} proposed and analyzed a two-point gradient method in Banach spaces based on the Landweber iteration and an extrapolation
strategy. Their method can be written as the following: \begin{equation}\label{(1.7)}\begin{split}
\Im_k^{\delta}=\gamma_k^{\delta}+\lambda_k^{\delta}(\gamma_k^{\delta}-\gamma_{k-1}^{\delta})\\ w_k^{\delta}=\arg\min_{w\in U}\big\{\varphi(w)-\langle \Im_{k}^{\delta}, w\rangle\big\},\hspace{-13mm}\\\gamma_{k+1}^{\delta}=\Im_k^{\delta}-\upsilon_k^{\delta}F'(w_k^{\delta})^*J_s^V((F(w_k^{\delta})-v^{\delta}),\hspace{-25mm}
\end{split}
\end{equation}
with suitably chosen combination parameters $\lambda_k^{\delta}$ and step sizes $\upsilon_k^{\delta}$. Then a discrepancy principle has been incorporated to terminate the  iteration and the approximate solution 
is calculated as follows: $$u_k^{\delta}=\arg\min_{u\in U}\big\{\varphi(u)-\langle \gamma_{k}^{\delta}, u\rangle\big\}.$$
Observe that \autoref{(1.7)} becomes Landweber iteration \autoref{(1.5)} for $\lambda_k^{\delta}=0$. And \autoref{(1.7)}  with $\lambda_k^{\delta}=\frac{k}{k+\varsigma}$ is a refined
version of the  Landweber iteration via Nesterov acceleration \cite{Jin3}.

In this paper, we use the general uniformly convex penalty term $\varphi$ to propose a variant of the two point gradient method \autoref{(1.7)} by incorporating  the iteratively regularized Landweber  iteration scheme \autoref{(1.4)} with an  extrapolation step. The proposed method takes the form \begin{equation}\label{(1.8)}\begin{split}
\Im_k^{\delta}=\gamma_k^{\delta}+\lambda_k^{\delta}(\gamma_k^{\delta}-\gamma_{k-1}^{\delta})\hspace{10mm}\\ w_k^{\delta}=\arg\min_{w\in U}\big\{\varphi(w)-\langle \Im_{k}^{\delta}, w\rangle\big\},\hspace{-4mm}\\\gamma_{k+1}^{\delta}=(1-\alpha_k)\Im_k^{\delta}-\upsilon_k^{\delta}F'(w_k^{\delta})^*J_s^V((F(w_k^{\delta})-v^{\delta})+\alpha_k\Im_0,\hspace{-45mm}
\end{split}
\end{equation}
 with suitably chosen combination parameters $\lambda_k^{\delta}$, step sizes $\upsilon_k^{\delta}$ and $\{\alpha_k\}\in [0,1]$. Further, we also incorporate a discrepancy principle  to terminate the  iteration and the approximate solution $u_k^{\delta}$  will be calculated by solving the following problem: $$\arg\min_{u\in U}\big\{\varphi(u)-\langle \gamma_{k}^{\delta}, u\rangle\big\}.$$ Observe that  \autoref{(1.8)} with $\{\alpha_k\}=\{0\}$ becomes method of the form \autoref{(1.7)}. In the case of Hilbert spaces $U$ and $V$ with $\{\alpha_k\}=\{0\}$ and $\varphi(z)=\|z\|^2/2$, \autoref{(1.8)} is nothing but the two-point gradient method introduced in \cite{Hub}. Unlike \cite{Hub}, the method \autoref{(1.8)} is not only suitable for inverse problems in Banach spaces, but  also permits   to use total variation like functions as well as $L^1$ functionals. However, due to non-Hilbertian structures of $U$ and $V$ and non-smoothness of $\varphi$, we need to incorporate various geometrical properties of Banach spaces and tools from convex analysis to study the convergence analysis.

 In order to study the convergence analysis of our method \autoref{(1.8)}, we need to employ certain conditions on the combination parameters $\lambda_k^{\delta}$ and the step sizes $\upsilon_k^{\delta}$. We adapt the discrete backtracking search (DBTS) algorithm considered in \cite{Hub, Zhong} to find the non-trivial combination parameters for our method. In our analysis, we incorporate standard assumptions such as tangential cone condition \cite{Hanke}, conditional Lipschitz stability of the inverse problem \cite{Hoop1, Hoop2}, boundedness of the approximation of Fr\'echet derivative etc. 
 
 It is worth to mention that on taking $\lambda_k^{\delta}\neq 0$ in \autoref{(1.8)}, the method becomes a novel  variant of  the two-point gradient method based on iteratively regularized Landweber iteration method. Therefore, our contribution in this paper is twofold.  We discuss the convergence analysis of our novel scheme \autoref{(1.8)} based on iteratively regularized Landweber iteration method together with an extrapolation step.
  Complementary, we also get the  convergence analysis of the following variant of iteratively regularized Landweber iteration method which is not  discussed in the literature yet: \begin{equation*}\begin{split}
\hspace{0mm}\gamma_{k+1}^{\delta}=(1-\alpha_k)\gamma_k^{\delta}-\upsilon_k^{\delta}F'(w_k^{\delta})^*J_s^V((F(w_k^{\delta})-v^{\delta})+\alpha_k\gamma_0,\hspace{-12mm}\\ u_k^{\delta}=\arg\min_{u\in U}\big\{\varphi(u)-\langle \gamma_{k}^{\delta}, u\rangle\big\}.\hspace{30mm}
\end{split}
\end{equation*}

The paper is organized in the following manner.  In Section $2$,  some preliminaries from
convex analysis are given. In Section $3$, we exhibit our novel two-point gradient method with a general uniformly convex penalty term together with its  detailed convergence analysis. The section $4$ includes the discussion on the  choices of  combination parameters. In particular, we discuss the modified  DBTS algorithm in this section. { We also compare our method with the method of \cite{Zhong} in this section}. The Section $5$ comprises discussion on a severe ill-posed problem on which our novel method is applicable.
 Finally in Section $6$, we conclude the paper.

\section{Preliminaries}
In this section, we discuss some basic concepts related to convex analysis and Banach spaces. Most of these details can be found in \cite{Z\u alinscu}.
Let $\|\cdot\|, U^*$ denote the norm and dual space, respectively of a Banach space $U$. We write $\langle \gamma, u\rangle=\gamma(u)$ for the duality mapping for a given $\gamma\in U^*$ and $u\in U$. 
In our analysis, we consider the convex function $u\to \frac{\|u\|^s}{s}$ for $1<s<\infty$. Its subgradient at $u$ is defined as\begin{equation*}
J_s^U(u):=\big\{\gamma\in U^*:\|\gamma\|=\|u\|^{s-1}\ \text{and}\ \langle \gamma, u\rangle=\|u\|^s\big\}.
\end{equation*}
This subgradient gives the set valued duality mapping $J_s^U:U\to 2^{U^*}$ of $U$ with the gauge function $x\to x^{s-1}$. We require the duality mapping $J_s^U$ to be single valued in our analysis. So, in order to achieve this we define the notion of uniform smoothness. A Banach space $U$ is said to be uniformly smooth if $\lim_{x\to0}\frac{\rho_U(x)}{x}=0$, where $\rho_U(x)$ is the smoothness modulus defined as\begin{equation*}
\rho_U(x):=\sup\{\|\tilde{u}+u\|+\|\tilde{u}-u\|-2:\|\tilde{u}\|=1, \|u\|\leq x\}.
\end{equation*} 
It can be seen that if $U$ is uniformly smooth, then duality mappings $J_s^U$ are uniformly bounded on bounded sets and single valued for every $1<s<\infty$. Some examples of uniformly smooth Banach spaces are $\ell^s$, $W^{k, s}$ and $L^s$.

 Let $\partial\varphi(u)$ denotes the subdifferential of a convex function $\varphi(u):U\to(-\infty, \infty]$ at $u\in U$. Mathematically, we have \begin{equation*}
\partial\varphi(u):=\{\gamma\in U^*: \varphi(\tilde{u})-\varphi(u)-\langle \gamma, \tilde{u}-u\rangle\geq 0\ \text{for all}\ \tilde{u}\in U\}.
\end{equation*}
Let $D(\varphi):=\{u\in U:\varphi(u)<\infty\}$ be the effective domain of $\varphi$ and let $D(\partial\varphi):=\{u\in D(\varphi):\partial\varphi(u)\neq\emptyset\}$. A proper convex function $\varphi : U \to (-\infty, \infty]$ is said to be uniformly convex if there
exists a  function $\Psi : [0,\infty) \to [0,\infty)$   such that
\begin{equation}\label{(2.1)}
\varphi(\tau \tilde{u}+(1-\tau)u)+\tau(1-\tau)\Psi(\|u-\tilde{u}\|)\leq \tau\varphi(\tilde{u})+(1-\tau)\varphi(u), \ \forall\ \tilde{u}, u\in U, 
\end{equation}
where $\Psi$ is strictly increasing and satisfying the boundary condition $\Psi(0)=0$, and $\tau\in [0, 1]$.  The Bregman distance between $\tilde{u}\in U$ and $u\in U$, induced by $\varphi$ in the direction $\gamma\in \partial\varphi(u)$ at $u$ is defined as \begin{equation}\label{(2.2)}
\mathcal{D}_{\gamma}\varphi(\tilde{u}, u)=\varphi(\tilde{u})-\varphi(u)-\langle \gamma,  \tilde{u}-u\rangle,
\end{equation}
One can easily see that $\mathcal{D}_{\gamma}\varphi(\tilde{u}, u)\geq 0$ and it satisfies the following three point identity\begin{equation}\label{(2.3)}
\mathcal{D}_{\gamma_2}\varphi(u, u_2)-\mathcal{D}_{\gamma_1}\varphi(u, u_1)=\mathcal{D}_{\gamma_2}\varphi(u_1, u_2)+\langle \gamma_2-\gamma_1, u_1-u\rangle
\end{equation}
for all $u_1, u_2\in D(\partial\varphi), u\in D(\varphi),$ and $\gamma_i\in \partial\varphi(u_i), i=1, 2$.

To this end, let us move on to recall  the concept of  Legendre–Fenchel conjugate. The Legendre–Fenchel conjugate of a proper lower semi-continuous convex function $\varphi:U\to(-\infty, \infty]$ is defined as $$\varphi^*(\gamma):=\sup_{u\in U}\big\{\langle \gamma, u\rangle-\varphi(u)\big\}, \ \ \gamma\in U^*.$$Observe that $\varphi^*(\gamma)$ is proper, convex and lower semi-continuous. Further, if $U$ is a reflexive Banach space, then \begin{equation}\label{(2.4)}
\gamma\in \partial\varphi(u)\Longleftrightarrow u\in \partial\varphi^*(\gamma)\Longleftrightarrow \varphi(u)+\varphi^*(\gamma)=\langle\gamma, u\rangle.
\end{equation}
Next, if we consider $\Psi(t)=c_0t^p$ for $p>1$ and $c_0>0$ in \autoref{(2.1)}, then the function $\varphi$ is called $p$-convex. It can be proved that $\varphi$ is $p$-convex if and only if \begin{equation} \label{(2.5)}
\mathcal{D}_{\gamma}\varphi(\tilde{u}, u)\geq c_0\|u-\tilde{u}\|^p, \ \ \forall \tilde{u}\in U, u\in D(\partial\varphi), \gamma\in \partial \varphi(u).
\end{equation}
Now we recall some  properties of the Legendre–Fenchel conjugate $\varphi^*$. For $p>1$, if $\varphi$ is $p$-convex, then from \cite[Corollary 3.5.11]{Z\u alinscu} it is known that $D(\varphi^*)=U^*, \varphi^*$ is Fr\'echet differentiable and $\nabla\varphi^*:U^*\to U$ fulfills\begin{equation}\label{(2.6)}
\|\nabla\varphi^*(\gamma_1)-\nabla\varphi^*(\gamma_2)\|\leq \bigg(\frac{\|\gamma_1-\gamma_2\|}{2c_0}
\bigg)^{\frac{1}{p-1}}, \ \ \forall\ \gamma_1, \gamma_2\in U^*.\end{equation}
Consequently, \autoref{(2.5)} implies that \begin{equation}\label{(2.7)}
u=\nabla \varphi^*(\gamma)\Longleftrightarrow\gamma\in \partial\varphi(u)\Longleftrightarrow u=\arg\min_{w\in U}\big\{\varphi(w)-\langle\gamma, w\rangle\big\}.
\end{equation}
Now if $\varphi$ is $p$-convex for $p>1$ and $u, \tilde{u}\in D(\partial\varphi), \gamma \in \partial\varphi(u), \tilde{\gamma}\in \partial\varphi(\tilde{u})$, then \autoref{(2.2)} and \autoref{(2.4)} lead to \begin{equation}
\label{(2.8)}\begin{split}
\mathcal{D}_{\gamma}(\tilde{u}, u)=\varphi^*(\gamma)-\varphi^*(\tilde{\gamma})-\langle \gamma-\tilde{\gamma}, \nabla \varphi^*(\tilde{\gamma})\rangle\\ =\int_0^1 \langle \gamma-\tilde{\gamma}, \nabla \varphi^*(\tilde{\gamma}+t(\gamma-\tilde{\gamma}))-\nabla \varphi^*(\tilde{\gamma})\rangle\, dt.\hspace{-20mm}
\end{split}
\end{equation}
Combining the estimate\autoref{(2.8)} with \autoref{(2.6)} to obtain
\begin{equation}
\label{(2.9)}\begin{split}
\mathcal{D}_{\gamma}(\tilde{u}, u)\leq \|\gamma-\tilde{\gamma}\|\int_0^1 \|\nabla \varphi^*(\tilde{\gamma}+t(\gamma-\tilde{\gamma}))-\nabla \varphi^*(\tilde{\gamma})\|\, dt \\ \leq\frac{1}{p^*(2c_0)^{p^*-1}} \|\gamma-\tilde{\gamma}\|^{p^*}, \ \text{where}
\ \frac{1}{p^*}+\frac{1}{p}=1.\hspace{7mm}
\end{split}
\end{equation}
\section{Convergence analysis of  the novel two point gradient method}
Throughout this section, we assume that $F$ satisfies \autoref{1.1}, data in \autoref{1.1} is attainable, $U$ is a reflexive Banach space and $V$ is a uniformly smooth Banach space. Let $\varphi:U\to(-\infty, \infty)$ denotes a general convex function which will be used as a penalty term. In order to exhibit our results, we need to have  certain assumptions accumulated in the following subsection.
\subsection{Assumptions}
\begin{enumerate}
\item The operator $F$ is weakly closed on its domain $D(F)$.
\item The function $\varphi:U\to(-\infty, \infty)$ is a proper, $p$-convex with $p>1$, weak lower semi-continuous such that \autoref{(2.5)} holds for some $c_0>0$.
\item There exist $u_0\in U$ and $\gamma_0\in \partial\varphi(u_0)$ such that \autoref{1.1} has a solution $u^*\in D(\varphi)$ with \begin{equation*}
\mathcal{D}_{\gamma_0}\varphi(u^*, u_0)\leq c_0\epsilon^p,
\end{equation*}
and $B(u_0, 3\epsilon)\subset D(F)$, where $B(u_0, \epsilon)$ denotes the closed ball of radius $\epsilon>0$ around $u_0$.
\item The inversion has the following 
Lipschitz type stability (cf. \cite{Hoop1, Hoop2}), i.e., there exists a constant $\mathcal{C} > 0$ such that
\begin{equation*}
\mathcal{D}_{\gamma}\varphi( \tilde{u}, u)\leq \mathcal{C}\|F(\tilde{u})-F(u)\|^p \ \ \forall u, \tilde{u}\in B(u_0, 3\epsilon)
\end{equation*}
for $\gamma\in \partial\varphi(u)$.
\item There exists a family of bounded linear operators $\{L(u):U\to V\}_{u\in B(u_0, 3\epsilon)\cap D(\varphi)}$ such that the function $$u\to L(u)\ \ \text{is continuous on}\ B(u_0, 3\epsilon)\cap D(\varphi).$$ Further, there exists a constant $\eta$ with $0\leq \eta<1$ such that the tangential cone condition\begin{equation*}
\|F(\tilde{u}) - F(u) - L(u)(\tilde{u}-u)\|\leq  \eta \|F(\tilde{u})- F(x)\|
\end{equation*}
holds for all $u, \tilde{u}\in B(u_0, 3\epsilon)\cap D(\varphi).$
\item There exists a constant $C_0>0$ such that for all $u\in B(u_0, 3\epsilon)$, we have \begin{equation*}
\|L(u)\|_{U\to V}\leq  C_0.\end{equation*}
\end{enumerate}
We define $u^{\dagger}$ to be solution of \autoref{1.1} which satisfy \begin{equation}\label{(3.1)}
\mathcal{D}_{\gamma_0}\varphi(u^{\dagger}, u_0)=\min_{u\in D(F)\cap D(\varphi)}\big\{\mathcal{D}_{\gamma_0}\varphi(u, u_0): F(u)=v\big\}.
\end{equation}
By employing the weak closedness of $F$, $p$-convexity and weak lower semicontinuity of $\varphi$ and reflexivity of $U$, it can be shown that such a $u^{\dagger}$ exists.
Further, the following lemma guarantees the unique solution of \autoref{1.1} satisfying \autoref{(3.1)}. See, \cite[Lemma 3.2]{Jin2} for its proof.\begin{lem}
There exists a unique solution of $\autoref{1.1}$ satisfying $\autoref{(3.1)}$, provided the assumptions discussed in Subsection $3.1$ hold.
\end{lem}
Note that the point $(4)$ of Subsection $3.1$ is not required for proving the last lemma. Let us now move on to define our scheme in the following subsection.
\subsection{Novel iteration scheme}
In this subsection, we formulate a variant of the two-point gradient method as discussed in the introduction. In this variant, the $p$-convex  function $\varphi$ induces a penalty term. Let us assume that $u_{-1}^{\delta}=u_0^{\delta}:=u_0\in U$ and $\gamma_{-1}^{\delta}=\gamma_0^{\delta}:=\gamma_0\in\partial\varphi(u_0)$ as the initial guess, and $\tau>1$ be a given number. For $n\geq 0$, define\begin{equation}\label{(3.2)}\begin{split}
\Im_k^{\delta}=\gamma_k^{\delta}+\lambda_k^{\delta}(\gamma_k^{\delta}-\gamma_{k-1}^{\delta}),\hspace{10mm}\\ w_k^{\delta}=\nabla\varphi^*(\Im_k^{\delta}),
\hspace{27mm}\\\gamma_{k+1}^{\delta}=(1-\alpha_k)\Im_k^{\delta}-\upsilon_k^{\delta}L(w_k^{\delta})^*J_s^V(r_k^{\delta})+\alpha_k\Im_0,\hspace{-25mm}\\ u_{k+1}^{\delta}=\nabla^*(\gamma_{k+1}^{\delta}),
\hspace{26mm}
\end{split}
\end{equation}
 with suitably chosen combination parameters $\lambda_k^{\delta}$, $F(w_k^{\delta})-v^{\delta}=r_k^{\delta}$,   $\{\alpha_k\}\in [0,1]$, step sizes  $\upsilon_k^{\delta}$ which will be defined shortly and the duality mapping $J_s^V:V\to V^*$ with the gauge function $t\to t^{s-1}, 1<s<\infty$. Note that due to uniform smoothness of $V$, $J_s^V$ is continuous as well as single-valued.
 
Let us denote $t_k^{\delta}:=\|\Im_k^{\delta}-\Im_0\|$. Observe that both $t_k^{\delta}$ and $r_k^{\delta}$ are available after the second step of $k^{\text{th}}$ iteration of \autoref{(3.2)}.   The step sizes  $\upsilon_k^{\delta}$ considered in \autoref{(3.2)} are required in the third step which means they can be defined in terms of $t_k^{\delta}$ and $r_k^{\delta}$  as follows:
\begin{equation}\label{(3.3)}
\upsilon_k^{\delta}=\begin{cases}\min\left\{
\dfrac{\dfrac{1}{2}\bigg(\vartheta_1^{p^*-1}
\|r_k^{\delta}\|^{s}-\vartheta_{2, k}(t_k^{\delta})^{p^*}\bigg)^{\frac{1}{p^*-1}}}{\|
L(w_k^{\delta})^*J_s^V(r_k^{\delta})\|^{p}}, \vartheta_3\|r_k^{\delta}\|^{p-s} \right\}\ \ \text{if}\ \|r_k^{\delta}\|>\tau\delta\\\hspace{20mm}0 \hspace{72mm}\ \text{if}\ \|r_k^{\delta}\|\leq\tau\delta,
\end{cases}
\end{equation}
where the positive constant $\vartheta_1$ and the  sequence $\{\vartheta_{2, k}\}$ are such that \begin{equation*}
\vartheta_{2, k}(t_k^{\delta})^{p^*}\leq \bar{\vartheta_{2}}^{p^*-1}\|r_k^{\delta}\|^{s}\leq \vartheta_1^{p^*-1}\|r_k^{\delta}\|^{s},
\end{equation*}
and $\bar{\vartheta_{2}}>0$,  $\vartheta_3>0$. In the following remark we discuss how to choose the constant $\vartheta_1$ and the  sequence $\{\vartheta_{2, k}\}$ in \autoref{(3.3)}
\begin{rema}
In order to determine the constant $\vartheta_1$ and an element $\vartheta_{2, k}$ of the sequence $\{\vartheta_{2, k}\}$ for $k^{\text{th}}$ iteration, we use the available values  $t_k^{\delta}$ and $\|r_k^{\delta}\|$.  For an  arbitrary but fixed positive real number $\vartheta_1$ (the involvement of constant $\vartheta_1$ will be clear when we discuss Proposition $3.3$), we take a fixed positive real number $\vartheta_{2, k}$ such that $\vartheta_1^{p^*-1}
\|r_k^{\delta}\|^{s}-\vartheta_{2, k}(t_k^{\delta})^{p^*}>0$. The sequence $\{\vartheta_{2, k}\}$  is essentially required here as in case $\|r_k^{\delta}\|\to 0$, the term $\vartheta_1^{p^*-1}
\|r_k^{\delta}\|^{s}-\varrho (t_k^{\delta})^{p^*}$ becomes negative after certain stage for any positive constant $\varrho$. This will make the step size negative which is not the case. \end{rema}
 Clearly, in our method previous two iterations are required at each step. It is worth to mention that the $p$-convex function $\varphi$ in our method can be a general non-smooth penalty function. This feature allows to reconstruct solutions having certain features such as discontinuities and sparsity.

Further,  let $\alpha_k$ in \autoref{(3.2)}  be such that whenever $\upsilon_k^{\delta}, t_k^{\delta} \neq 0$, it satisfies \begin{equation}\label{A}
\alpha_k \leq\min\bigg\{ \vartheta_4\upsilon_k^{\delta}\|F(w_k^{\delta})-v^{\delta}\|^{s-1}(t_k^{\delta})^{-1},\ 2^{\frac{1-p^*}{p^*}}(\vartheta_{2, k}\upsilon_k^{\delta})^{\frac{1}{p^*}}\bigg\},
\end{equation}

for some  positive  constant $\vartheta_4$. If $\upsilon_k^{\delta}= 0$, then $\alpha_k$ can be a arbitrary sequence. Note that the terms $\vartheta_{2, k}$, $\upsilon_k^{\delta}$, $t_k$ $r_k^{\delta}$ are available before the third step of our scheme \autoref{(3.2)}, so they can be utilized to obtain $\alpha_k$.

As usual, we employ the discrepancy principle with respect to $w_k^{\delta}$ in order to properly terminate our novel scheme \autoref{(3.2)}. By employing this principle,  the method would provide a  useful approximate solution to \autoref{1.1}. For $\tau > 1$, we stop the iteration 
after $k_{\delta}$ steps, where the integer $k_{\delta}:=k(\delta, v^{\delta})$   is  such that\begin{equation}\label{(3.4)}
\|F(w^{\delta}_{k_{\delta}})-v^{\delta}\|\leq \tau\delta<\|F(w^{\delta}_{k})-v^{\delta}\|, \ \ 0\leq k<k_{\delta}
\end{equation}
and use $u^{\delta}_{k_{\delta}}$ as the approximate solution.

\subsection{Convergence Analysis} In this subsection, we perform the convergence analysis of our novel scheme \autoref{(3.2)}. In this regard,  let us begin by  recalling an important result from \cite[Proposition 3.6]{Jin2} which would be  employed later on to prove the convergence of the iterates in the presence of exact data. \begin{prop}
Let the assumptions of Subsection $3.1$  hold $($except point $4)$ and $\varphi:U\to(-\infty,\infty]$ be a proper, uniformly convex and lower semi-continuous   function. Let $\{u_k\} \subset B(u_0, 2\epsilon) \cap  D(\varphi)$ and $\{\gamma_k\} \subset U^*$ be such that the following hold\begin{enumerate}[(i)]
\item $\gamma_k\in \partial\varphi(u_k)$ for all $k$.
\item for any solution $\hat{u}$ of $\autoref{1.1}$ in $B(u_0, 2\epsilon) \cap  D(\varphi)$ the sequence $\{\mathcal{D}_{\gamma_k}\varphi(\hat{u}, u_k)\}$ is monotonically decreasing.
\item $\lim_{k\to\infty}\|F(u_k)-v\|=0$.
\item there is a subsequence $\{k_n\}$ with $k_n\to\infty$ such that for any solution $\hat{u}$ of $\autoref{1.1}$ in $B(u_0, 2\epsilon) \cap  D(\varphi)$ there holds $$\lim_{l\to\infty}\sup_{n\geq l}|\langle\gamma_{k_n}-\gamma_{k_l}, u_{k_n}-\hat{u}\rangle|
=0.$$\end{enumerate}
Then there exists a solution $\bar{u}$ of $\autoref{1.1}$ in  $B(u_0, 2\epsilon) \cap  D(\varphi)$ such that
$$\lim_{k\to\infty}\mathcal{D}_{\gamma_k}\varphi(\bar{u}, u_k)=0.$$
\end{prop} Now, in order to study the convergence analysis, first we show   the monotonocity of the Bregman distance $\mathcal{D}_{\gamma_k^{\delta}}\varphi(\hat{u}, u_k^{\delta})$ with respect to $k$ for $0\leq k\leq k_{\delta}$, where $\hat{u}$ is any solution of \autoref{1.1} in $B(u_0, 2\epsilon)\cap D(\varphi)$. In this regard, let us first obtain the estimates for $\mathcal{D}_{\Im_k^{\delta}}\varphi(\hat{u}, w_k^{\delta})-\mathcal{D}_{\gamma_k^{\delta}}\varphi(\hat{u}, u_k^{\delta})$ and  $\mathcal{D}_{\gamma_{k+1}^{\delta}}\varphi(\hat{u}, u_{k+1}^{\delta})-\mathcal{D}_{\Im_k^{\delta}}\varphi(\hat{u}, w_k^{\delta})$ in the following proposition under certain assumptions.\begin{prop}
Let $V$ be uniformly smooth, $U$ be reflexive and assumptions of Subsection $3.1$ hold. Then, for any solution $\hat{u}\in B(u_0, 2\epsilon)\cap D(\varphi)$ of $\autoref{1.1}$, we have
\begin{equation}\label{(3.5)}\begin{split}
\mathcal{D}_{\Im_k^{\delta}}\varphi(\hat{u}, w_k^{\delta})-\mathcal{D}_{\gamma_{k}^{\delta}}\varphi(\hat{u}, u_{k}^{\delta})\hspace{80mm}\\ \leq 
\lambda_k^{\delta}\Theta_k+\frac{\lambda_k^{\delta}}{p^*(2c_0)^{p^*-1}} \|\gamma_{k}^{\delta}-\gamma_{k-1}^{\delta}\|^{p^*}+\frac{(\lambda_{k}^{\delta})^{p^*}}{p^*(2c_0)^{p^*-1}} \|\gamma_{k}^{\delta}-\gamma_{k-1}^{\delta}\|^{p^*}. \end{split}
\end{equation}
Further if $w_k^{\delta}\in B(u_0, 3\epsilon)$ then 
\begin{equation}\label{(3.6)}\begin{split}
\mathcal{D}_{\gamma_{k+1}^{\delta}}\varphi(\hat{u}, u_{k+1}^{\delta})-\mathcal{D}_{\Im_k^{\delta}}\varphi(\hat{u}, w_k^{\delta})\hspace{80mm} \\ \leq 
 \bigg[\bigg(\frac{\mathcal{C}}{c_0}\bigg)^{\frac{1}{p}}\vartheta_4
+(1+\eta)\bigg]\upsilon_k^{\delta}\|F(w_k^{\delta})-v^{\delta}\|^{s-1}\delta\hspace{50mm}\\-\bigg[1-\bigg(\frac{\mathcal{C}}{c_0}\bigg)^{\frac{1}{p}}\vartheta_4-\eta-\frac{1}{p^*(2c_0)^{p^*-1}} \big(\vartheta_1^{p^*-1}\big)\bigg]\upsilon_k^{\delta}\|F(w_k^{\delta})-v^{\delta}\|^{s},\end{split}
\end{equation}
where \begin{equation}\label{(3.7)}
\Theta_k:=\mathcal{D}_{\gamma_k^{\delta}}\varphi(\hat{u}, u_k^{\delta})-\mathcal{D}_{\gamma_{k-1}^{\delta}}\varphi(\hat{u}, u_{k-1}^{\delta}).
\end{equation}
\end{prop}
\begin{proof}
To derive \autoref{(3.5)}, let us first obtain an estimate for $\langle \Im_k^{\delta}-\gamma_k^{\delta}, u_k^{\delta}-\hat{u}\rangle$. By using the definition of $\Im_k^{\delta}$, three point identity \autoref{(2.3)}, \autoref{(3.7)} and \autoref{(2.9)}, we have $$\langle\Im_k^{\delta}-\gamma_k^{\delta}, u_k^{\delta}-\hat{u}\rangle=\lambda_k^{\delta}\langle \gamma_k^{\delta}-\gamma_{k-1}^{\delta}, u_k^{\delta}-\hat{u}\rangle\hspace{90mm}$$ $$=\lambda_k^{\delta}\big(\mathcal{D}_{\gamma_k^{\delta}}\varphi(\hat{u}, u_k^{\delta})-\mathcal{D}_{\gamma_{k-1}^{\delta}}\varphi(\hat{u}, u_{k-1}^{\delta})+\mathcal{D}_{\gamma_{k-1}^{\delta}}\varphi(u_k^{\delta}, u_{k-1}^{\delta})\big)$$ $$=\lambda_k^{\delta}\Theta_k +\lambda_k^{\delta}\mathcal{D}_{\gamma_{k-1}^{\delta}}\varphi(u_k^{\delta}, u_{k-1}^{\delta})\hspace{48mm}$$ \begin{equation}\label{(3.8)}
\leq \lambda_k^{\delta}\Theta_k+\frac{\lambda_k^{\delta}}{p^*(2c_0)^{p^*-1}} \|\gamma_{k}^{\delta}-\gamma_{k-1}^{\delta}\|^{p^*}.\hspace{35mm}
\end{equation}
Again use the three point identity \autoref{(2.3)}, \autoref{(2.9)} and definition of $\Im_k^{\delta}$ to obtain
$$\mathcal{D}_{\Im_k^{\delta}}\varphi(\hat{u}, w_k^{\delta})-\mathcal{D}_{\gamma_{k}^{\delta}}\varphi(\hat{u}, u_{k}^{\delta})=\langle\Im_k^{\delta}-\gamma_k^{\delta}, u_k^{\delta}-\hat{u}\rangle+\mathcal{D}_{\Im_{k}^{\delta}}\varphi(u_k^{\delta}, w_{k}^{\delta})\hspace{30mm}$$ $$\leq  \langle\Im_k^{\delta}-\gamma_k^{\delta}, u_k^{\delta}-\hat{u}\rangle+\frac{1}{p^*(2c_0)^{p^*-1}} \|\Im_{k}^{\delta}-\gamma_{k}^{\delta}\|^{p^*}\hspace{-40mm}$$ $$=  \langle\Im_k^{\delta}-\gamma_k^{\delta}, u_k^{\delta}-\hat{u}\rangle+\frac{(\lambda_{k}^{\delta})^{p^*}}{p^*(2c_0)^{p^*-1}} \|\gamma_{k}^{\delta}-\gamma_{k-1}^{\delta}\|^{p^*}.\hspace{-44mm}$$
Plugging the estimate \autoref{(3.8)} in above inequality to obtain $$\mathcal{D}_{\Im_k^{\delta}}\varphi(\hat{u}, w_k^{\delta})-\mathcal{D}_{\gamma_{k}^{\delta}}\varphi(\hat{u}, u_{k}^{\delta})\leq 
\lambda_k^{\delta}\Theta_k+\frac{\lambda_k^{\delta}}{p^*(2c_0)^{p^*-1}} \|\gamma_{k}^{\delta}-\gamma_{k-1}^{\delta}\|^{p^*}+\frac{(\lambda_{k}^{\delta})^{p^*}}{p^*(2c_0)^{p^*-1}} \|\gamma_{k}^{\delta}-\gamma_{k-1}^{\delta}\|^{p^*},$$ which is the required estimate \autoref{(3.5)}. Now, we prove the second part. For this, we start with the definition of $\gamma_{k+1}^{\delta}$ in \autoref{(3.2)}, according to which $$\|\gamma_{k+1}^{\delta}-\Im_k^{\delta}\|^{p^*}=\|\alpha_k\Im_k^{\delta}+\upsilon_k^{\delta}L(w_k^{\delta})^*J_s^V(r_k^{\delta})-\alpha_k\Im_0\|^{p^*}\hspace{50mm}$$ \begin{equation}\label{(3.9)}
\hspace{4mm} \leq 2^{p^*-1}\big((\upsilon_k^{\delta})^{p^*}\|L(w_k^{\delta})^*J_s^V(r_k^{\delta})\|^{p^*}+\alpha_k^{p^*}\|\Im_k^{\delta}-\Im_0\|^{p^*}\big).\end{equation}
Further, from the definition of step size in \autoref{(3.3)}, we have $$(\upsilon_k^{\delta})^{p^*-1}\leq\dfrac{\dfrac{1}{2^{p^*-1}}\big(\vartheta_1^{p^*-1}
\|r_k^{\delta}\|^{s}-\vartheta_{2,k}(t_k^{\delta})^{p^*}\big)}{\|
L(w_k^{\delta})^*J_s^V(r_k^{\delta})\|^{p(p^*-1)}}$$
which means \begin{equation}\label{(3.10)}
(\upsilon_k^{\delta})^{p^*-1}\|
L(w_k^{\delta})^*J_s^V(r_k^{\delta})\|^{p^*}\leq \dfrac{1}{2^{p^*-1}}\big(\vartheta_1^{p^*-1}
\|r_k^{\delta}\|^{s}-\vartheta_{2,k}(t_k^{\delta})^{p^*}\big).
\end{equation}
Substituting \autoref{(3.10)} in \autoref{(3.9)} to reach at\begin{equation}
 \label{(3.11)} \|\gamma_{k+1}^{\delta}-\Im_k^{\delta}\|^{p^*}\leq \upsilon_k^{\delta}\vartheta_1^{p^*-1}
\|r_k^{\delta}\|^{s}+\big(2^{p^*-1}\alpha_k^{p^*}- \vartheta_{2,k}\upsilon_k^{\delta}\big)(t_k^{\delta})^{p^*}.\end{equation}
To deduce the estimate \autoref{(3.6)}, three point identity \autoref{(2.3)} and \autoref{(2.9)} imply that $$\mathcal{D}_{\gamma_{k+1}^{\delta}}\varphi(\hat{u}, u_{k+1}^{\delta})-\mathcal{D}_{\Im_k^{\delta}}\varphi(\hat{u}, w_k^{\delta})=\mathcal{D}_{\gamma_{k+1}^{\delta}}\varphi(w_k^{\delta}, u_{k+1}^{\delta})+\langle\gamma_{k+1}^{\delta}-\Im_k^{\delta}, w_k^{\delta}-\hat{u}\rangle\hspace{14mm}$$\begin{equation}\label{(3.12)}
 \hspace{48mm}\leq \frac{1}{p^*(2c_0)^{p^*-1}} \|\gamma_{k+1}^{\delta}-\Im_{k}^{\delta}\|^{p^*}+\langle\gamma_{k+1}^{\delta}-\Im_k^{\delta}, w_k^{\delta}-\hat{u}\rangle.\end{equation}
Estimate for the first term of right side of inequality \autoref{(3.12)} has been already  deduced in \autoref{(3.11)}. So, let us deduce an estimate for the second term. For this, we use the definition of $\gamma_{k+1}^{\delta}$,  definition of duality mapping, \autoref{1.2} and \autoref{(3.3)} to derive that $$\langle\gamma_{k+1}^{\delta}-\Im_k^{\delta}, w_k^{\delta}-\hat{u}\rangle\hspace{110mm}$$ $$=-\langle \alpha_k\Im_k^{\delta}+\upsilon_k^{\delta}L(w_k^{\delta})^*J_s^V(F(w_k^{\delta})-v^{\delta})-\alpha_k\Im_0,w_k^{\delta}-\hat{u}\rangle\hspace{28mm} $$ $$=-\langle \alpha_k(\Im_k^{\delta}-\Im_0),w_k^{\delta}-\hat{u}\rangle  -\upsilon_k^{\delta}\langle J_s^V(F(w_k^{\delta})-v^{\delta}),L(w_k^{\delta})(w_k^{\delta}-\hat{u})\rangle\hspace{10mm} $$  $$=-\langle \alpha_k(\Im_k^{\delta}-\Im_0),w_k^{\delta}-\hat{u}\rangle  -\upsilon_k^{\delta}\langle J_s^V(F(w_k^{\delta})-v^{\delta}),v^{\delta}-
F(w_k^{\delta})-L(w_k^{\delta})(\hat{u}-w_k^{\delta})\rangle\hspace{-16mm} $$ $$-\upsilon_k^{\delta}\langle J_s^V(F(w_k^{\delta})-v^{\delta}),F(w_k^{\delta})-v^{\delta}\rangle\hspace{40mm}$$ $$\leq \alpha_kt_k^{\delta}\|w_k^{\delta}-\hat{u}\|+\upsilon_k^{\delta}\|F(w_k^{\delta})-v^{\delta}\|^{s-1}\big(\delta+\|v-
F(w_k^{\delta})-L(w_k^{\delta})(\hat{u}-w_k^{\delta})\|\big)\hspace{-8mm}$$  $$-\upsilon_k^{\delta} \|F(w_k^{\delta})-v^{\delta}\|^s.\hspace{65mm}$$Apply
points $(4)$, $(5)$ of assumptions in Subsection $3.1$ (as $w_k^{\delta}\in B(u_0, 3\epsilon)$ by assumption and $\hat{u}\in B(u_0, 3\epsilon)$ due to point $3$ of assumptions in Subsection $3.1$ and \autoref{(2.5)}) after applying \autoref{(2.5)}, and then \autoref{A}  in the last inequality to reach at
 $$\langle\gamma_{k+1}^{\delta}-\Im_k^{\delta}, w_k^{\delta}-\hat{u}\rangle\hspace{120mm}$$ $$\leq \alpha_kt_k^{\delta}\bigg(\frac{\mathcal{C}}{c_0}\bigg)^{\frac{1}{p}}\|F(w_k^{\delta})-F(\hat{u})\|+\upsilon_k^{\delta}\|F(w_k^{\delta})-v^{\delta}\|^{s-1}\big(\delta+\eta\|
F(w_k^{\delta})-v\|\big)\hspace{5mm}$$  $$-\upsilon_k^{\delta} \|F(w_k^{\delta})-v^{\delta}\|^s\hspace{75mm}$$
$$\leq \bigg(\frac{\mathcal{C}}{c_0}\bigg)^{\frac{1}{p}}\vartheta_4\upsilon_k^{\delta}\|F(w_k^{\delta})-v^{\delta}\|^{s-1}\big(\|F(w_k^{\delta})-v^{\delta}\|+\delta\big)+\upsilon_k^{\delta}\|F(w_k^{\delta})-v^{\delta}\|^{s-1}\big(\delta(1+\eta)\hspace{-13mm}$$ $$+\eta\|
F(w_k^{\delta})-v^{\delta}\|\big)  -\upsilon_k^{\delta} \|F(w_k^{\delta})-v^{\delta}\|^s\hspace{30mm}$$ $$= \bigg[\bigg(\frac{\mathcal{C}}{c_0}\bigg)^{\frac{1}{p}}\vartheta_4
+(1+\eta)\bigg]\upsilon_k^{\delta}\|F(w_k^{\delta})-v^{\delta}\|^{s-1}\delta\hspace{50mm}$$ $$-\bigg[1-\bigg(\frac{\mathcal{C}}{c_0}\bigg)^{\frac{1}{p}}\vartheta_4-\eta\bigg]\upsilon_k^{\delta}\|F(w_k^{\delta})-v^{\delta}\|^{s}.$$
Substituting this and \autoref{(3.11)} in \autoref{(3.12)} to arrive at  $$\mathcal{D}_{\gamma_{k+1}^{\delta}}\varphi(\hat{u}, u_{k+1}^{\delta})-\mathcal{D}_{\Im_k^{\delta}}\varphi(\hat{u}, w_k^{\delta})\hspace{100mm}$$ $$ \leq \frac{1}{p^*(2c_0)^{p^*-1}} \big(\upsilon_k^{\delta}\vartheta_1^{p^*-1}
\|F(w_k^{\delta})-v^{\delta}\|^{s}+\big(2^{p^*-1}\alpha_k^{p^*}-\vartheta_{2,k}\upsilon_k^{\delta}\big)(t_k^{\delta})^{p^*}\big)\hspace{34mm}$$ $$+ \bigg[\bigg(\frac{\mathcal{C}}{c_0}\bigg)^{\frac{1}{p}}\vartheta_4
+(1+\eta)\bigg]\upsilon_k^{\delta}\|F(w_k^{\delta})-v^{\delta}\|^{s-1}\delta-\bigg[1-\bigg(\frac{\mathcal{C}}{c_0}\bigg)^{\frac{1}{p}}\vartheta_4-\eta\bigg]\upsilon_k^{\delta}\|F(w_k^{\delta})-v^{\delta}\|^{s}.$$
This estimate with the choice of $\alpha_k$ in \autoref{A} is the estimate \autoref{(3.6)}.
\end{proof}
Till now, we have only obtained the estimates for any arbitrary $k$ in Proposition $3.2$. Let us now
 employ the discrepany principle in the results of Proposition $3.2$. For that, observe that from the definition of $\vartheta_k^{\delta}$,  \autoref{(3.3)} and \autoref{(3.4)}, we have $\vartheta_k^{\delta}\tau\delta\leq \vartheta_k^{\delta}\|F(w_k^{\delta})-v^{\delta}\|$. Plugging this in \autoref{(3.6)} to obtain\begin{equation}\label{(3.13)}\begin{split}
\mathcal{D}_{\gamma_{k+1}^{\delta}}\varphi(\hat{u}, u_{k+1}^{\delta})-\mathcal{D}_{\Im_k^{\delta}}\varphi(\hat{u}, w_k^{\delta})\leq -\vartheta_5\upsilon_k^{\delta}\|F(w_k^{\delta})-v^{\delta}\|^{s},\end{split}
\end{equation}
where\begin{equation}\label{(3.14)}
 \vartheta_5=1-\bigg(\frac{\mathcal{C}}{c_0}\bigg)^{\frac{1}{p}}\vartheta_4-\eta-\frac{\vartheta_1^{p^*-1}}{p^*(2c_0)^{p^*-1}} -\frac{\big(\frac{\mathcal{C}}{c_0}\big)^{\frac{1}{p}}\vartheta_4
+(1+\eta)}{\tau}.\end{equation}
We choose the constants $\tau$ (sufficiently large), and $\vartheta_1, \vartheta_2, \eta$ (all three sufficiently small) such that $\vartheta_5>0$ (cf. Remark 3.2). Now, with the definition of $\Theta_k$ in \autoref{(3.7)},  \autoref{(3.13)}   and \autoref{(3.5)}, we have
$$\Theta_{k+1}=\mathcal{D}_{\gamma_{k+1}^{\delta}}\varphi(\hat{u}, u_{k+1}^{\delta})-\mathcal{D}_{\gamma_{k}^{\delta}}\varphi(\hat{u}, u_{k}^{\delta})\hspace{69mm}$$ $$\leq\mathcal{D}_{\Im_k^{\delta}}\varphi(\hat{u}, w_k^{\delta})-\mathcal{D}_{\gamma_{k}^{\delta}}\varphi(\hat{u}, u_{k}^{\delta})-\vartheta_5\upsilon_k^{\delta}\|F(w_k^{\delta})-v^{\delta}\|^{s}\hspace{24mm}$$
$$\leq \lambda_k^{\delta}\Theta_k+\frac{\lambda_k^{\delta}+(\lambda_{k}^{\delta})^{p^*}}{p^*(2c_0)^{p^*-1}} \|\gamma_{k}^{\delta}-\gamma_{k-1}^{\delta}\|^{p^*}-\vartheta_5\upsilon_k^{\delta}\|F(w_k^{\delta})-v^{\delta}\|^{s}.\hspace{12mm}$$
So, we have proved the following:\begin{prop}
Let $V$ be uniformly smooth, $U$ be reflexive and assumptions of Subsection $3.1$ be satisfied. Further, let $\vartheta_5$ in $\autoref{(3.14)}$ be a positive constant and $w_k^{\delta}\in B(u_0, 3\epsilon)$. Then, for any solution $\hat{u}\in B(u_0, 2\epsilon)\cap D(\varphi)$ of $\autoref{1.1}$, we have \begin{equation}\label{(3.15)}
\Theta_{k+1}\leq \lambda_k^{\delta}\Theta_k+\frac{\lambda_k^{\delta}+(\lambda_{k}^{\delta})^{p^*}}{p^*(2c_0)^{p^*-1}} \|\gamma_{k}^{\delta}-\gamma_{k-1}^{\delta}\|^{p^*}-\vartheta_5\upsilon_k^{\delta}\|F(w_k^{\delta})-v^{\delta}\|^{s}.
\end{equation}
\end{prop}
Let us discuss about the requirement of  constant  $\vartheta_5$ in \autoref{(3.14)} to be positive in the following remark.
\begin{rema}
The constant $\vartheta_5$ in \autoref{(3.14)} depend on the following four variable constants, $\tau, \eta, \vartheta_4,$ and  $\vartheta_1$. Let us discuss how to choose these constants so that $\vartheta_5$  becomes positive. 
Since $\tau$ can be taken arbitrary large, the fraction $\frac{(\frac{\mathcal{C}}{c_0})^{\frac{1}{p}}\vartheta_4
+(1+\eta)}{\tau}$ can be neglected in comparison to $1$. The constant $\vartheta_4$ first  appeared in choice of $\alpha_k$ in \autoref{A} and is arbitrary. So, $\vartheta_4$ can  be taken as a small number. The constant $\eta$ is clearly less than $1$ (cf. point $(5)$ of assumption in Subsection $3.1$). Finally, the constant $\vartheta_4$ is also arbitrary which has been  intentionally introduced in \autoref{(3.3)} and can be taken very small. Therefore,  we conclude that $\vartheta_5$ can be  positive for wisely chosen constants.
\end{rema}

Note that  we have not yet discussed any conditions required to be satisfied by combination parameters $\lambda_k^{\delta}$ (see \autoref{(3.2)}) in our analysis. So, in this regard, let  $\zeta>1$ be a constant  such that for all $k\geq 0$, following two inequalities hold \begin{equation}\label{(3.16)}
\frac{\lambda_k^{\delta}+(\lambda_{k}^{\delta})^{p^*}}{p^*(2c_0)^{p^*-1}} \|\gamma_{k}^{\delta}-\gamma_{k-1}^{\delta}\|^{p^*}\leq \frac{\vartheta_5\upsilon_k^{\delta}}{\zeta}\|F(w_k^{\delta})-v^{\delta}\|^{s},
\end{equation}
\begin{equation}\label{(3.17)}
\frac{\lambda_k^{\delta}+(\lambda_{k}^{\delta})^{p^*}}{p^*(2c_0)^{p^*-1}} \|\gamma_{k}^{\delta}-\gamma_{k-1}^{\delta}\|^{p^*}\leq c_0\epsilon^p.
\end{equation}
Clearly, $\lambda_k^{\delta}=0$ satisfy the  inequalities \autoref{(3.16)}, \autoref{(3.17)}. The  technical discussion on  choosing the non-trivial $\lambda_k^{\delta}$ satisfying inequalities \autoref{(3.16)}, \autoref{(3.17)} is shifted to Section $4$. 

 Next, by engaging Propositions $3.2$ and $3.3$, we infact show that $w_k^{\delta}\in B(u_0, 3\epsilon)$ and monotonocity of the Bregman distance, i.e. $\Theta_k\leq 0$ with the choices of $\lambda_k^{\delta}$ satisfying \autoref{(3.16)}, \autoref{(3.17)}.
\begin{prop}
Let $V$ be uniformly smooth, $U$ be reflexive and assumptions of Subsection $3.1$ be satisfied. Further, let $\vartheta_5$ in $\autoref{(3.14)}$ be  a positive constant and $\lambda_k^{\delta}$ satisfy  $\autoref{(3.16)}$, $\autoref{(3.17)}$. Then\begin{enumerate}[(i)]
\item $w_k^{\delta}\in B(u_0, 3\epsilon)$ for $k\geq 0$.
\item $u_k^{\delta}\in B(u_0, 2\epsilon)$ for $k\geq 0$.
\end{enumerate}
Moreover, if $\hat{u}\in B(u_0, 2\epsilon)\cap D(\varphi)$ is any solution of $\autoref{1.1}$, then $\Theta_k\leq 0$. 
\end{prop}
\begin{proof}
Observe that with the initial choices $u_{-1}^{\delta}=u_0^{\delta}=u_0$ and $\gamma_{-1}^{\delta}=\gamma_0^{\delta}=\gamma_0\in\partial\varphi(u_0)$, $w_0^{\delta}=\nabla \varphi^*(\Im_0^{\delta})=\nabla \varphi^*(\gamma_0)=u_0$, parts $(i)$ and $(ii)$ are trivial. We prove the results $(i)$ and $(ii)$ via induction hypothesis. So, to this end, let $(i)$ and $(ii)$ hold for all integers less than or equal to a positive integer $r$. This means Proposition $3.3$ is valid for $w_r^{\delta}$ which gives
\begin{equation*}
\Theta_{r+1}\leq \lambda_r^{\delta}\Theta_r+\frac{\lambda_r^{\delta}+(\lambda_{r}^{\delta})^{p^*}}{p^*(2c_0)^{p^*-1}} \|\gamma_{r}^{\delta}-\gamma_{r-1}^{\delta}\|^{p^*}-\vartheta_5\upsilon_r^{\delta}\|F(w_r^{\delta})-v^{\delta}\|^{s}.
\end{equation*}
Further, by induction as $\Theta_r\leq 0$ and $\lambda_r^{\delta}\geq 0$, above inequality implies that 
\begin{equation*}
\Theta_{r+1}\leq \frac{\lambda_r^{\delta}+(\lambda_{r}^{\delta})^{p^*}}{p^*(2c_0)^{p^*-1}} \|\gamma_{r}^{\delta}-\gamma_{r-1}^{\delta}\|^{p^*}-\vartheta_5\upsilon_r^{\delta}\|F(w_r^{\delta})-v^{\delta}\|^{s}.
\end{equation*}
Incorporating \autoref{(3.16)} in above inequality yields 
\begin{equation}\label{C}
\Theta_{r+1}\leq \frac{\vartheta_5\upsilon_r^{\delta}}{\zeta} \|F(w_r^{\delta})-v^{\delta}\|^{s}-\vartheta_5\upsilon_r^{\delta}\|F(w_r^{\delta})-v^{\delta}\|^{s}\leq 0
\end{equation}
since $\zeta >1$. Thus, we have proved that $\Theta_k\leq 0$ for all $k$. Consequently, by taking $\hat{u}=u^{\dagger}$ and repeatedly applying the argument $\Theta_k\leq 0$, we get\begin{equation}\label{(3.18)}
 \mathcal{D}_{\gamma_{r+1}^{\delta}}\varphi(u^{\dagger}, u_{r+1}^{\delta})\leq\mathcal{D}_{\gamma_{r}^{\delta}}\varphi(u^{\dagger}, u_{r}^{\delta})\leq \cdots\leq \mathcal{D}_{\gamma_{0}^{\delta}}\varphi(u^{\dagger}, u_{0}^{\delta}).\end{equation}Plugging the estimate from point $(3)$ of assumptions in Subsection $3.1$ and \autoref{(2.5)} in \autoref{(3.18)} to reach at 
\begin{equation*}
c_0\|u_{r+1}^{\delta}-u^{\dagger}\|^p\leq 
\mathcal{D}_{\gamma_{r+1}^{\delta}}\varphi(u^{\dagger}, u_{r+1}^{\delta})\leq c_0\epsilon^p.
\end{equation*}
Again apply point $(3)$ of assumptions in Subsection $3.1$ and \autoref{(2.5)} with $\hat{u}=u^{\dagger}$ to see that \begin{equation}\label{(3.19)}
c_0\|u_0-u^{\dagger}\|^p\leq 
\mathcal{D}_{\gamma_{0}}\varphi(u^{\dagger}, u_{0}^{\delta})\leq c_0\epsilon^p.
\end{equation}
From the last two estimates, we have that 
\begin{equation*}
\|u_{r+1}^{\delta}-u_0\|\leq  \|u_{r+1}^{\delta}-u^{\dagger}\|+\|u^{\dagger}-u_0\|\leq 2\epsilon.
\end{equation*}
Thus, $u_{r+1}^{\delta}\in B(u_0, 2\epsilon)$ which means that proof of part $(ii)$ is complete. Now we move on to prove  part $(i)$. For that, observe that \autoref{(3.17)} and \autoref{(3.5)} provide  the estimate \begin{equation*}\begin{split}
\mathcal{D}_{\Im_{r+1}^{\delta}}\varphi(u^{\dagger}, w_{r+1}^{\delta})-\mathcal{D}_{\gamma_{r+1}^{\delta}}\varphi(u^{\dagger}, u_{r+1}^{\delta})\hspace{80mm}\\ \leq 
\lambda_{r+1}^{\delta}\Theta_{r+1}+\frac{\lambda_{r+1}^{\delta}}{p^*(2c_0)^{p^*-1}} \|\gamma_{r+1}^{\delta}-\gamma_{r}^{\delta}\|^{p^*}+\frac{(\lambda_{r+1}^{\delta})^{p^*}}{p^*(2c_0)^{p^*-1}} \|\gamma_{r+1}^{\delta}-\gamma_{r}^{\delta}\|^{p^*} \end{split}
\end{equation*}\begin{equation*}
\leq \lambda_{r+1}^{\delta}\Theta_{r+1}+c_0\epsilon^p.\hspace{61mm} 
\end{equation*}
This with  \autoref{(3.18)}, \autoref{(3.19)} and the assertion $\Theta_{r+1}\leq 0$ further provides the estimate
\begin{equation*} 
\mathcal{D}_{\Im_{r+1}^{\delta}}\varphi(u^{\dagger}, w_{r+1}^{\delta})\leq \mathcal{D}_{\gamma_{0}^{\delta}}\varphi(u^{\dagger}, u_{0}^{\delta})+c_0\epsilon^p\leq 2c_0\epsilon^p.\end{equation*}
Plugging \autoref{(2.5)} in above to deduce that\begin{equation*}
c_0\|w_{r+1}^{\delta}-u^{\dagger}\|^p\leq \mathcal{D}_{\Im_{r+1}^{\delta}}\varphi(u^{\dagger}, w_{r+1}^{\delta})\leq 2c_0\epsilon^p.
\end{equation*}
This estimate and \autoref{(3.19)} imply that\begin{equation*}
\|w_{r+1}^{\delta}-u_0\|\leq 2^{\frac{1}{p}}\epsilon+\epsilon\leq 3\epsilon,
\end{equation*}
since $2^{\frac{1}{p}}<2$. Therefore, $w_{r+1}^{\delta}\in B(u_0, 3\rho)$ which means that  proof of part $(i)$ and that of proposition is complete.
\end{proof}
We have incorporated the discrepancy principle \autoref{(3.4)} in our analysis. Through the  following proposition, we show that the stopping index  $k_{\delta}$ chosen via discrepancy principle is finite.
\begin{prop}
With the assumptions of Proposition $3.4$, we have
\begin{equation}\label{(3.20)}
\sum_{r=0}^k\upsilon_r^{\delta}\|F(w_r^{\delta})-v^{\delta}\|^{s}\leq \vartheta_5^{-1}\frac{\zeta}{\zeta-1} \mathcal{D}_{\gamma_{0}^{\delta}}\varphi(\hat{u}, u_{0}^{\delta}).
\end{equation}
Moreover, if the stopping index  $k_{\delta}$ is chosen via discrepancy principle $\autoref{(3.4)}$, then it  is finite.
\end{prop}
\begin{proof}
Since we have considered the assumptions of Proposition $3.4$, all of its results are applicable in this result. From \autoref{C}, for $r\geq 0$, we have $$\vartheta_5\upsilon_r^{\delta}\|F(w_r^{\delta})-v^{\delta}\|^{s}- \frac{\vartheta_5\upsilon_r^{\delta}}{\zeta} \|F(w_r^{\delta})-v^{\delta}\|^{s} \leq\mathcal{D}_{\gamma_{r}^{\delta}}\varphi(\hat{u}, u_{r}^{\delta})-\mathcal{D}_{\gamma_{r+1}^{\delta}}\varphi(\hat{u}, u_{r+1}^{\delta}).$$Hence, for any integer $k$, summing above from $r=0$ to $k$ yields \begin{equation*}
\vartheta_5\bigg(1-\frac{1}{\zeta}\bigg)\sum_{r=0}^k\upsilon_r^{\delta}\|F(w_r^{\delta})-v^{\delta}\|^{s}\leq  \sum_{r=0}^k\big(\mathcal{D}_{\gamma_{r}^{\delta}}\varphi(\hat{u}, u_{r}^{\delta})-\mathcal{D}_{\gamma_{r+1}^{\delta}}\varphi(\hat{u}, u_{r+1}^{\delta})\big)\hspace{14mm}
\end{equation*}
$$\hspace{35mm}= \mathcal{D}_{\gamma_{0}^{\delta}}\varphi(\hat{u}, u_{0}^{\delta})-\mathcal{D}_{\gamma_{k+1}^{\delta}}\varphi(\hat{u}, u_{k+1}^{\delta})$$ $$\hspace{5mm}\leq \mathcal{D}_{\gamma_{0}^{\delta}}\varphi(\hat{u}, u_{0}^{\delta}).$$ Thus, last inequality is the desired  estimate \autoref{(3.20)}. Next, we  show that the stopping index $k_{\delta}$ is finite. To see this, let on the contrary that  $k_{\delta}$ is infinite. Therefore, due to \autoref{(3.4)}, $\|r_k^{\delta}\|>\tau\delta$ for all $k\geq 0$. Consequently,  from the definition of $\upsilon_k^{\delta}$ in \autoref{(3.3)}, it can be easily seen that\begin{equation}\label{(3.21)}
\upsilon_k^{\delta}=\min\left\{
\dfrac{\dfrac{1}{2}\bigg(\vartheta_1^{p^*-1}
\|r_k^{\delta}\|^{s}-\vartheta_{2, k}(t_k^{\delta})^{p^*}\bigg)^{\frac{1}{p^*-1}}}{\|
L(w_k^{\delta})^*J_s^V(r_k^{\delta})\|^{p}}, \vartheta_3\|r_k^{\delta}\|^{p-s} \right\}.
\end{equation}
By utilizing the point $(6)$ of assumptions in Subsection $3.1$, observe that $$\|
L(w_k^{\delta})^*J_s^V(r_k^{\delta})\|^{p}\leq C_0^p\|J_s^V(r_k^{\delta})\|^{p}=C_0^p\|r_k^{\delta}\|^{p(s-1)}.$$
This with \autoref{(3.21)} leads to the inequality
\begin{equation*}
\upsilon_k^{\delta}\geq \min\left\{
\dfrac{\dfrac{1}{2}\bigg(\vartheta_1^{p^*-1}
\|r_k^{\delta}\|^{s}-\vartheta_{2, k}(t_k^{\delta})^{p^*}\bigg)^{\frac{1}{p^*-1}}}{C_0^p\|r_k^{\delta}\|^{p(s-1)}}, \vartheta_3\|r_k^{\delta}\|^{p-s} \right\}.
\end{equation*}
Above with the choice of $\vartheta_{2, k}$ gives \begin{equation*}
\upsilon_k^{\delta}\geq \min\left\{
\dfrac{(\vartheta_1^{p^*-1}-\bar{\vartheta_2}^{p^*-1})^{\frac{1}{p^*-1}}\|r_k^{\delta}\|^{\frac{s}{p^*-1}}}{2C_0^p\|r_k^{\delta}\|^{p(s-1)}}, \vartheta_3\|r_k^{\delta}\|^{p-s} \right\}
\end{equation*} 
\begin{equation}\label{G}
\hspace{5mm}= \min\left\{
\dfrac{(\vartheta_1^{p^*-1}-\bar{\vartheta_2}^{p^*-1})^{\frac{1}{p^*-1}}
\|r_k^{\delta}\|^{p-s}}{2C_0^p}, \vartheta_3\|r_k^{\delta}\|^{p-s} \right\},
\end{equation}
since $\frac{s}{p^*-1}-p(s-1)=p-s$.
Therefore, this with \autoref{(3.4)} and \autoref{(3.20)} yields \begin{equation*} \vartheta_5^{-1}\frac{\zeta}{\zeta-1} \mathcal{D}_{\gamma_{0}^{\delta}}\varphi(\hat{u}, u_{0}^{\delta})\geq
\min\left\{
\dfrac{(\vartheta_1^{p^*-1}-\bar{\vartheta_2}^{p^*-1})^{\frac{1}{p^*-1}}
}{2C_0^p}, \vartheta_3 \right\}\ \sum_{r=0}^k\|F(w_r^{\delta})-v^{\delta}\|^{p}
\end{equation*} $$\hspace{22mm}\geq \min\left\{
\dfrac{(\vartheta_1^{p^*-1}-\bar{\vartheta_2}^{p^*-1})^{\frac{1}{p^*-1}}
}{2C_0^p}, \vartheta_3 \right\}\ \sum_{r=0}^k(\tau\delta)^{p}$$$$\hspace{20mm}= \min\left\{
\dfrac{(\vartheta_1^{p^*-1}-\bar{\vartheta_2}^{p^*-1})^{\frac{1}{p^*-1}}
}{2C_0^p}, \vartheta_3 \right\}\ k(\tau\delta)^{p}.$$
Since $k$ is arbitrary, right side of above inequality can be arbitrary large, however left side is some fixed finite number. Thus, we have arrived at a contradiction which means  that $k_{\delta}$ is finite. This completes the proof.\end{proof}
Now, we establish a convergence  result for our  novel  iteration scheme \autoref{(3.2)} in which we show that  in the presence of  precise data, iterates of \autoref{(3.2)} necessarily converges to a solution of \autoref{1.1}. In order to see this, we assume that $\delta=0$ and consider the scheme \autoref{(3.2)} by omitting the superscript $\delta$ from all the parameters in which it is involved. We remark that all the parameters and constants mentioned in Subsection $3.2$ would have  same meaning except $\upsilon_k$ in \autoref{(3.3)} which we redefine as 
\begin{equation}\label{(3.22)}
\upsilon_k=\begin{cases}\min\left\{
\dfrac{\dfrac{1}{2}\bigg(\vartheta_1^{p^*-1}
\|r_k\|^{s}-\vartheta_{2, k}(t_k)^{p^*}\bigg)^{\frac{1}{p^*-1}}}{\|
L(w_k)^*J_s^V(r_k)\|^{p}}, \vartheta_3\|r_k\|^{p-s} \right\}\ \ \text{if}\ r_k\neq 0\\\hspace{20mm}0 \hspace{72mm}\ \text{if}\ r_k=0.
\end{cases}
\end{equation}
We are now ready to discuss the convergence result for scheme \autoref{(3.2)} in the noise free case with the aid of Proposition of $3.1$. To this end, let \begin{equation}\label{(3.23)}
 \vartheta_6=1-\bigg(\frac{\mathcal{C}}{c_0}\bigg)^{\frac{1}{p}}\vartheta_4-\eta-\frac{\big(\vartheta_1^{p^*-1}\big)}{p^*(2c_0)^{p^*-1}}.\end{equation}
 
\begin{thm}
Let $V$ be uniformly smooth, $U$ be reflexive and assumptions of Subsection $3.1$ be satisfied. Further, let  $\lambda_k$ satisfy  $\autoref{(3.16)}$, $\autoref{(3.17)}$ $($with $\delta=0)$ and $\vartheta_6$ defined in $\autoref{(3.23)}$ be positive. Moreover, assume that the sequences $\{\alpha_k\}$ and $\{\lambda_k\}$ are such that\begin{equation}\label{F}
\sum_{k=0}^{\infty}\alpha_k\|\gamma_0-\gamma_k\|<\infty,
\end{equation}
\begin{equation}\label{D}
\sum_{k=0}^{\infty}\lambda_k\|\gamma_k-\gamma_{k-1}\|<\infty.
\end{equation}
Then, there exists a $\bar{u}$ which satisfy $\autoref{1.1}$ in $B(u_0, 2\epsilon)\cap D(\varphi)$ such that $$\lim_{k\to\infty}\mathcal{D}_{\gamma_k}\varphi(\bar{u}, u_k)=0, \ \text{and}\ \ \lim_{k\to\infty}\|u_k-\bar{u}\|=0.$$
\end{thm}
\begin{proof}
By definition, we know that $u_k=\nabla\varphi^*(\gamma_k)$ which means $\gamma_k\in \partial\varphi(u_k)$. This means part $(i)$ of Proposition $3.1$ is satisfied. Now as  $\vartheta_6>0$, all the assumptions of Proposition $3.4$ are satisfied. Therefore, from Proposition $3.4$ we can see that $\Theta_k\leq 0$, i.e. $\mathcal{D}_{\gamma_{k}}\varphi(\hat{u}, u_{k})\leq \mathcal{D}_{\gamma_{k-1}}\varphi(\hat{u}, u_{k-1})$ for all $k$. Thus, $(ii)$ of Proposition $3.1$ is also satisfied. Due to monotonicity of the sequence $\{\mathcal{D}_{\gamma_{k}}\varphi(\hat{u}, u_{k})\}$, let $\lim_{k\to\infty}\mathcal{D}_{\gamma_{k}}\varphi(\hat{u}, u_{k})=\varrho$. This limit exists uniquely as the sequence  $\{\mathcal{D}_{\gamma_{k}}\varphi(\hat{u}, u_{k})\}$ is bounded below.

Since the assumptions of Proposition $3.5$ are only those of Proposition $3.4$, we have \begin{equation}\label{(3.24)}
\sum_{k=0}^{\infty}\upsilon_k\|F(w_k)-v\|^{s}< \infty.
\end{equation} Moreover, from \autoref{G} we know that \begin{equation*}
\upsilon_k\geq   \min\left\{
\dfrac{(\vartheta_1^{p^*-1}-\bar{\vartheta_2}^{p^*-1})^{\frac{1}{p^*-1}}
\|r_k\|^{p-s}}{2C_0^p}, \vartheta_3\|r_k\|^{p-s} \right\}.
\end{equation*}
Multiply both sides by $\|r_k\|^{s}$ to get 
\begin{equation*}
\upsilon_k\|F(w_k)-v\|^s\geq   \min\left\{
\dfrac{(\vartheta_1^{p^*-1}-\bar{\vartheta_2}^{p^*-1})^{\frac{1}{p^*-1}}}{2C_0^p}, \vartheta_3 \right\}\|F(w_k)-v\|^p.
\end{equation*}
Also, by definition of $\upsilon_k$ in \autoref{(3.22)}, we know that $$\upsilon_k\|F(w_k)-v\|^s\leq \vartheta_3\|F(w_k)-v\|^p.$$
Consequently, from the last two inequalities and \autoref{(3.24)},  it follows  that \begin{equation}\label{(3.25)}
\sum_{k=0}^{\infty}\|F(w_k)-v\|^{p}< \infty\implies \lim_{k\to\infty}\|F(w_k)-v\|=0. 
\end{equation} 
Further, we have$$\|F(u_k)-v\|\leq \|F(u_k)-F(w_k)\|+\|F(w_k)-v\|.$$Above implies that $\lim_{k\to\infty}\|F(u_k)-v\|=0$ provided $\lim_{k\to\infty}\|F(u_k)-F(w_k)\|=0$.
Due to points $(5)$ and $(6)$ of assumptions in Subsection $3.1$, we have $$\|F(u_k)-F(w_k)\|\leq \|F(u_k)-F(w_k)-L(w_k)(u_k-w_k)\|
+\|L(w_k)(u_k-w_k)\|$$ $$\ \leq \eta \|F(u_k)-F(w_k)\|+\|L(w_k)(u_k-w_k)\|$$ $$\ \leq \eta \|F(u_k)-F(w_k)\|+C_0\|u_k-w_k\|\hspace{9mm}$$
which means 
$$\|F(u_k)-F(w_k)\|\leq \frac{C_0}{1-\eta}\|u_k-w_k\|.$$
Applying the definitions of  $u_k$ and $w_k$ and plugging \autoref{(2.6)} in above to deduce that $$\|F(u_k)-F(w_k)\|\leq \frac{C_0}{1-\eta}\|\nabla\varphi^*(\gamma_k)-\nabla\varphi^*(\Im_k)\|\hspace{20mm}$$ $$\hspace{14mm}\leq  \frac{C_0}{1-\eta}\frac{1}{(2c_0)^{p^*-1}}\|\gamma_k-\Im_k\|^{p^*-1}$$ \begin{equation}\label{(3.26)}
 \hspace{20mm}=\frac{C_0}{1-\eta}\bigg(\frac{\lambda_k}{2c_0}\bigg)^{p^*-1}\|\gamma_k-\gamma_{k-1}\|^{p^*-1}.\end{equation}
With the choice of combination parameters  \autoref{(3.16)}, we have 
\begin{equation*}
\frac{\lambda_{k}^{p^*}}{p^*(2c_0)^{p^*-1}} \|\gamma_{k}-\gamma_{k-1}\|^{p^*}\leq \frac{\vartheta_5\upsilon_k}{\zeta}\|F(w_k)-v\|^{s},
\end{equation*}
which can also be written as
\begin{equation*}
\lambda_{k}^{p^*-1} \|\gamma_{k}-\gamma_{k-1}\|^{p^*-1}\leq \bigg(p^*(2c_0)^{p^*-1} \frac{\vartheta_5}{\zeta}\bigg)^{\frac{1}{p}}(\upsilon_k\|F(w_k)-v\|^{s})^\frac{1}{p}.
\end{equation*}
Incorporate the definition of $\upsilon_k$ from \autoref{(3.22)} to further reach at \begin{equation*}
\lambda_{k}^{p^*-1} \|\gamma_{k}-\gamma_{k-1}\|^{p^*-1}\leq \bigg(p^*(2c_0)^{p^*-1} \frac{\vartheta_5}{\zeta}\bigg)^{\frac{1}{p}}(\vartheta_3\|F(w_k)-v\|^{p})^\frac{1}{p}.
\end{equation*}
Plugging this in \autoref{(3.26)} yields \begin{equation}\label{E}
 \|F(u_k)-F(w_k)\|\leq \frac{C_0}{1-\eta} \bigg(p^* \frac{\vartheta_3\vartheta_5}{2c_0\zeta}\bigg)^{\frac{1}{p}}\|F(w_k)-v\|.\end{equation}Hence,  due to \autoref{(3.25)}, $\|F(u_k)-F(w_k)\|\to 0$ as $k\to\infty$. Thus $(iii)$ of Proposition $3.1$ also holds. 

From part $(iii)$ we know that $\|F(w_k)-v\|\to 0$ as $k\to\infty$. To this end, let us choose a sequence $\{k_n\}$ of integers which is strictly increasing by letting $k_0=0$ and $k_n$ be the first integer satisfying $k_n\geq k_{n-1}+1$ and $\|F(w_{k_n})-v\|\leq \|F(w_{k_{n-1}})-v\|.$ From this choice of sequence, for all $0\leq k<k_n$, we have that \begin{equation}\label{(3.27)}
\|F(w_{k_n})-v\|\leq \|F(w_{k})-v\|.
\end{equation}
Next, let us deduce an estimate for $\|L(w_k)(u_{k_n}-\hat{u})\|$ whenever $k<k_n$ which will be employed shortly. By using points $(5)$ and $(6)$ of assumptions in Subsection $3.1$ and \autoref{(3.27)}, we obtain  $$\|L(w_k)(u_{k_n}-\hat{u})\|\leq \|L(w_k)(u_{k_n}-w_k)\|+\|L(w_k)(w_k-\hat{u})\|\hspace{55mm}$$ $$\hspace{15mm}\leq \|F(u_{k_n})-F(w_k)-L(w_k)(u_{k_n}-w_k)\|+\|F(u_{k_n})-F(w_k)\|$$ $$ \hspace{10mm}+\ \|F(w_k)-v-L(w_k)(w_k-\hat{u})\|+\|F(w_k)-v\|$$
 $$\leq(\eta+1)\big(\|F(u_{k_n})-F(w_k)\|+\|F(w_k)-v\|\big)\hspace{13mm}$$ $$\leq(\eta+1)\big(\|F(u_{k_n})-v\|+2\|F(w_k)-v\|\big)\hspace{20mm}$$ $$\leq(\eta+1)\big(\|F(u_{k_n})-F(w_{k_n})\|+\|F(w_{k_n})-v\|+2\|F(w_k)-v\|\big)\hspace{-20mm}$$ \begin{equation}\label{(3.28)}
   \leq(\eta+1)\big(\|F(u_{k_n})-F(w_{k_n})\|+3\|F(w_k)-v\|\big)\hspace{9mm}\end{equation} 
For $k<k_n$ employ \autoref{E} and \autoref{(3.27)} to get $$\|F(u_{k_n})-F(w_{k_n})\|\leq \frac{C_0}{1-\eta} \bigg(p^* \frac{\vartheta_3\vartheta_5}{2c_0\zeta}\bigg)^{\frac{1}{p}}\|F(w_{k_n})-v\|\hspace{30mm}$$
\begin{equation*}\hspace{5mm}
\leq \frac{C_0}{1-\eta} \bigg(p^* \frac{\vartheta_3\vartheta_5}{2c_0\zeta}\bigg)^{\frac{1}{p}}\|F(w_{k})-v\|.
\end{equation*}
Substituting this in \autoref{(3.28)} to obtain \begin{equation}\label{(3.29)}
\|L(w_k)(u_{k_n}-\hat{u})\|\leq \bigg[\frac{C_0(1+\eta)}{1-\eta} \bigg(p^* \frac{\vartheta_3\vartheta_5}{2c_0\zeta}\bigg)^{\frac{1}{p}}+3(1+\eta)\bigg]\|F(w_{k})-v\|.
\end{equation}

Let us now start with the left hand side of the part $(iv)$ of Proposition $3.1$. From \autoref{(3.2)}, recall  that 
$$\gamma_{k+1}-\gamma_k=(1-\alpha_k)\lambda_k(\gamma_{k}-\gamma_{k-1})-\upsilon_k L(w_k)^*J_s^V(F(w_k)-v)+\alpha_k(\gamma_0-\gamma_k),$$ since $\gamma_0=\Im_0$.  From this, we have $$\big|\langle\gamma_{k_n}-\gamma_{k_l}, u_{k_n}-\hat{u}\rangle\big|=\sum_{k=k_l}^{k_n-1}\big|\langle \gamma_{k+1}-\gamma_k, u_{k_n}-\hat{u}\rangle\big|\hspace{60mm}$$ $$\hspace{10mm}\leq \sum_{k=k_l}^{k_n-1}(1-\alpha_k)\lambda_k\big|\langle \gamma_{k}-\gamma_{k-1}, u_{k_n}-\hat{u}\rangle\big|+\sum_{k=k_l}^{k_n-1}\upsilon_k\big|\langle L(w_k)^*J_s^V(F(w_k)-v),u_{k_n}-\hat{u}\rangle\big|$$ $$+\sum_{k=k_l}^{k_n-1}\alpha_k\big|\langle \gamma_0-\gamma_k, u_{k_n}-\hat{u}\rangle\big|$$  $$\hspace{10mm}\leq \sum_{k=k_l}^{k_n-1}(1-\alpha_k)\lambda_k\|\gamma_{k}-\gamma_{k-1}\|\  \|u_{k_n}-\hat{u}\|+\sum_{k=k_l}^{k_n-1}\upsilon_k\big|\langle J_s^V(F(w_k)-v),L(w_k)(u_{k_n}-\hat{u})\rangle\big|$$ $$+\sum_{k=k_l}^{k_n-1}\alpha_k\|\gamma_0-\gamma_k\| \ \|u_{k_n}-\hat{u}\|$$ $$\hspace{10mm}\leq \sum_{k=k_l}^{k_n-1}(1-\alpha_k)\lambda_k\|\gamma_{k}-\gamma_{k-1}\|\  \|u_{k_n}-\hat{u}\|+\sum_{k=k_l}^{k_n-1}\upsilon_k\|F(w_k)-v\|^{s-1}\|L(w_k)(u_{k_n}-\hat{u})\|$$ $$+\sum_{k=k_l}^{k_n-1}\alpha_k\|\gamma_0-\gamma_k\| \ \|u_{k_n}-\hat{u}\|,$$  where the  last inequality is obtained by using the property of $J_s^V$.   At this point, using the result that $u_{k_n}\in B(u_0, 2\epsilon)$, i.e. $u_{k_n}\in B(\hat{u}, 4\epsilon)$ and \autoref{(3.29)} in the last inequality to further reach at 
$$\big|\langle\gamma_{k_n}-\gamma_{k_l}, u_{k_n}-\hat{u}\rangle\big|\hspace{120mm}$$ \begin{equation}\label{(3.30)}
 \leq 4\epsilon\sum_{k=k_l}^{k_n-1}(1-\alpha_k)\lambda_k\|\gamma_{k}-\gamma_{k-1}\|\ +\vartheta_6\sum_{k=k_l}^{k_n-1}\upsilon_k\|F(w_k)-v\|^{s}+4\epsilon\sum_{k=k_l}^{k_n-1}\alpha_k\|\gamma_0-\gamma_k\|,\end{equation}
where $\vartheta_6=\frac{C_0(1+\eta)}{1-\eta} \big(p^* \frac{\vartheta_3\vartheta_5}{2c_0\zeta}\big)^{\frac{1}{p}}+3(1+\eta).$ From \autoref{C}, we know that $$\vartheta_5\upsilon_r\|F(w_r)-v\|^{s}-\frac{\vartheta_5\upsilon_r}{\zeta} \|F(w_r)-v\|^{s}\leq -\Theta_{r+1}.$$This with \autoref{(3.30)} gives the estimate $$\big|\langle\gamma_{k_n}-\gamma_{k_l}, u_{k_n}-\hat{u}\rangle\big|\hspace{120mm}$$ \begin{equation*}
 \leq 4\epsilon\sum_{k=k_l}^{k_n-1}(1-\alpha_k)\lambda_k\|\gamma_{k}-\gamma_{k-1}\|\ +4\epsilon\sum_{k=k_l}^{k_n-1}\alpha_k\|\gamma_0-\gamma_k\|\hspace{20mm}\end{equation*}\begin{equation*}
 \hspace{20mm}+\frac{\zeta \vartheta_6}{(\zeta-1)\vartheta_5}\sum_{k=k_l}^{k_n-1}
\big(D_{\gamma_k}\varphi(\hat{u}, u_k)-D_{\gamma_{k+1}}\varphi(\hat{u}, u_{k+1})\big) \end{equation*} \begin{equation*}
= 4\epsilon\sum_{k=k_l}^{k_n-1}(1-\alpha_k)\lambda_k\|\gamma_{k}-\gamma_{k-1}\|\ +4\epsilon\sum_{k=k_l}^{k_n-1}\alpha_k\|\gamma_0-\gamma_k\|\hspace{20mm}\end{equation*}\begin{equation*}
 \hspace{20mm}+\frac{\zeta \vartheta_6}{(\zeta-1)\vartheta_5}
\big(D_{\gamma_{k_l}}\varphi(\hat{u}, u_{k_l})-D_{\gamma_{k_n}}\varphi(\hat{u}, u_{k_n})\big). \end{equation*}
Hence, for the fixed $l$, we have $$\sup_{k\geq l}\big|\langle\gamma_{k_n}-\gamma_{k_l}, u_{k_n}-\hat{u}\rangle\big|\hspace{114mm}$$ \begin{equation*}
\leq 4\epsilon\sum_{k=k_l}^{k_n-1}(1-\alpha_k)\lambda_k\|\gamma_{k}-\gamma_{k-1}\|\ +4\epsilon\sum_{k=k_l}^{k_n-1}\alpha_k\|\gamma_0-\gamma_k\|
 +\frac{\zeta \vartheta_6}{(\zeta-1)\vartheta_5}
\big(D_{\gamma_{k_l}}\varphi(\hat{u}, u_{k_l})-\varrho)\end{equation*}
\begin{equation*}
\leq 4\epsilon\sum_{k=k_l}^{k_n-1}\lambda_k\|\gamma_{k}-\gamma_{k-1}\|\ +4\epsilon\sum_{k=k_l}^{k_n-1}\alpha_k\|\gamma_0-\gamma_k\|
 +\frac{\zeta \vartheta_6}{(\zeta-1)\vartheta_5}
\big(D_{\gamma_{k_l}}\varphi(\hat{u}, u_{k_l})-\varrho),\hspace{12mm} \end{equation*}
where we used the results that $\alpha_k\leq 1$ and $\lim_{k\to\infty}D_{\gamma_{k}}\varphi(\hat{u}, u_{k})=\varrho$ which is already discussed in the beginning of proof. Taking limit $l\to\infty$ in above and plug the estimates \autoref{F} and \autoref{D} to deduce that $$\lim_{l\to\infty}\sup_{k\geq l}\big|\langle\gamma_{k_n}-\gamma_{k_l}, u_{k_n}-\hat{u}\rangle\big|\leq \lim_{l\to\infty}\frac{\zeta \vartheta_6}{(\zeta-1)\vartheta_5}
\big(D_{\gamma_{k_l}}\varphi(\hat{u}, u_{k_l})-\varrho)=0.$$
Hence, $(iv)$ of Proposition $3.1$ also holds. Therefore, by the virtue of Proposition $3.1$ and \autoref{(2.5)}, desired result holds.
 \end{proof}

 Finally, we discuss our main result in which we show that our novel scheme \autoref{(3.2)} is a convergent regularization method if it is stopped via discrepancy principle \autoref{(3.4)}. In order to prove this, let us first  discuss a stability result which would be helpful in proving the main result.\begin{prop}
Let $V$ be uniformly smooth, $U$ be reflexive and assumptions of Subsection $3.1$ be satisfied. Further, let $\vartheta_5$ in $\autoref{(3.14)}$ be  positive and the combination parameters $\lambda_k^{\delta}$ satisfy $\autoref{(3.16)}$, $\autoref{(3.17)}$ and $\autoref{D}$. Moreover, let 
 $\lambda_k^{\delta}$  depend continuously on $\delta$ as $\delta\to 0$. Then for all $k\geq 0$ we have$$\Im_k^{\delta}\to \Im_k, \ \ \ \gamma_k^{\delta}\to\gamma_k,\ \ \ u_k^{\delta}\to u_k, \ \ \ w_k^{\delta}\to w_k\ \ \ \text{as}\ \ \  \delta\to 0.$$ 
 \end{prop}
 \begin{proof}
 For $k=0$, there is nothing to prove as the result holds trivially. Let us proceed with the induction and assume that  the result holds for all $0\leq k\leq n$. We need to show it for $k=n+1$ for which we consider the following two cases:
 
 Case-1: $r_n\neq 0$. Clearly, in this case for small $\delta>0$, we have $\|F(w_n)- v\|> \tau\delta$. Therefore, \autoref{(3.3)} becomes \begin{equation*}
\upsilon_n^{\delta}=\min\left\{
\dfrac{\dfrac{1}{2}\bigg(\vartheta_1^{p^*-1}
\|r_n^{\delta}\|^{s}-\vartheta_{2, n}(t_n^{\delta})^{p^*}\bigg)^{\frac{1}{p^*-1}}}{\|
L(w_n^{\delta})^*J_s^V(r_n^{\delta})\|^{p}}, \vartheta_3\|r_n^{\delta}\|^{p-s} \right\},
\end{equation*}  \begin{equation*}
\upsilon_n=\min\left\{
\dfrac{\dfrac{1}{2}\bigg(\vartheta_1^{p^*-1}
\|r_n\|^{s}-\vartheta_{2, n}(t_n)^{p^*}\bigg)^{\frac{1}{p^*-1}}}{\|
L(w_n)^*J_s^V(r_n)\|^{p}}, \vartheta_3\|r_n\|^{p-s} \right\}.
\end{equation*}
Again, we have two possibilities here. First one is if $L(w_n)^*J_s^V(r_n)=L(w_n)^*J_s^V(F(w_n)-v)\neq 0$. By induction hypothesis on $w_n^{\delta}$ it is easy to deduce that $\upsilon_n^{\delta}\to\upsilon_n$. Consequently,  due to induction hypothesis and  incorporating the continuity of $L$, $F$, $J_s^V$ and $\nabla\varphi^*$, we conclude that whenever $\delta\to 0$, $\Im_{n+1}^{\delta}\to \Im_{n+1}$,    $\gamma_{n+1}^{\delta}\to\gamma_{n+1}$, $u_{n+1}^{\delta}\to u_{n+1}$,  $w_{n+1}^{\delta}\to w_{n+1}$.  

Second possibility is if $L(w_n)^*J_s^V(r_n)=0$. In this case, for small $\delta>0$ we have 
 $$\upsilon_n^{\delta}=\vartheta_3\|r_n^{\delta}\|^{p-s}, \ \ \upsilon_n=\vartheta_3\|r_n\|^{p-s}.$$
 This means that $\upsilon_n^{\delta}\to\upsilon_n$. Rest part of the proof is similar to that of first possibility.

Case-2: $r_n=0$. Here, $\upsilon_n=0$. Consequently by the induction hypothesis $w_n^{\delta}\to w_n$ and   continuity of $F$, $\|F(w_n^{\delta})-v^{\delta}\|\to 0$ as $r_n=0$. Therefore, from \autoref{(3.2)} we have $$\gamma_{n+1}^{\delta}-\gamma_{n+1}=(1-\alpha_{n})(\Im_{n}^{\delta}-\Im_{n})-\upsilon_n^{\delta}L(w_n^{\delta})J_s^V(r_n^{\delta}).$$
 This with the induction hypothesis $\Im_{n}^{\delta}\to\Im_{n}$, point $(6)$ of assumptions in Subsection $3.1$ and definition of $\upsilon_n^{\delta}$ lead to the estimate $$\|\gamma_{n+1}^{\delta}-\gamma_{n+1}\|\leq (1-\alpha_{n})\|\Im_{n}^{\delta}-\Im_{n}\|+\upsilon_n^{\delta}C_0\|J_s^V(r_n^{\delta})\|\hspace{24mm}$$ $$\leq (1-\alpha_{n})\|\Im_{n}^{\delta}-\Im_{n}\|+\upsilon_n^{\delta}C_0\|r_n^{\delta}\|^{s-1}$$ $$\leq (1-\alpha_{n})\|\Im_{n}^{\delta}-\Im_{n}\|+\vartheta_3C_0\|r_n^{\delta}\|^{p-1} \to 0 \ \text{as}\ \delta\to 0.\hspace{-28mm}$$ Therefore $\gamma_{n+1}^{\delta}\to \gamma_{n+1}$. Also, we know that $$u_{n+1}^{\delta}=\nabla\varphi^*(\gamma_{n+1}^{\delta}), \ \ u_{n+1}=\nabla\varphi^*(\gamma_{n+1}).$$ Since $\gamma_{n+1}^{\delta}\to \gamma_{n+1}$, due to continuity of $\nabla\varphi^*$, we have that $u_{n+1}^{\delta}\to u_{n+1}$ as $\delta\to 0$. Due to \autoref{(3.2)} we deduce that $$\Im_{n+1}^{\delta}=\gamma_{n+1}^{\delta}+\lambda_{n+1}^{\delta}(\gamma_{n+1}^{\delta}-\gamma_{n}^{\delta})\to \Im_{n+1}\ \text{as}\ \delta\to 0,$$
where we used the assumption that  $\lambda_{n+1}^{\delta}$ depends continuously on $\delta$. 
 Finally, as $$w_{n+1}^{\delta}=\nabla\varphi^*(\Im_{n+1}^{\delta}), \ \ w_{n+1}=\nabla\varphi^*(\Im_{n+1}),$$ we have that  $\Im_{n+1}^{\delta}\to \Im_{n+1}$, due to continuity of $\nabla\varphi^*$. Thus, result.
 \end{proof}
We are now ready to give final main result of this section which renders our method \autoref{(3.2)} a regularizing one.
 \begin{thm}
Let $V$ be uniformly smooth, $U$ be reflexive and assumptions of Subsection $3.1$ be satisfied. Further, let $\vartheta_5$ in $\autoref{(3.14)}$ be positive and the combination parameters $\lambda_k^{\delta}$ satisfy $\autoref{(3.16)}$, $\autoref{(3.17)}$ and $\autoref{D}$. Moreover, let 
 $\lambda_k^{\delta}$  depend continuously on $\delta$ as $\delta\to 0$ and $\alpha_k$ satisfy $\autoref{F}$. Let the stopping rule $k_{\delta}$  be chosen such that $\autoref{(3.4)}$ be satisfied. Then there exists a solution $\bar{u}$ which satisfy $\autoref{1.1}$ in $B(u_0, 2\epsilon)\cap D(\varphi)$ such  that $$\lim_{\delta\to0}\mathcal{D}_{\gamma_{k_{\delta}}^{\delta}}\varphi(\bar{u}, u_{k_\delta}^{\delta})=0, \ \text{and}\ \ \lim_{\delta\to0}\|u_{k_{\delta}}^{\delta}-\bar{u}\|=0.$$
 \end{thm}
 \begin{proof}
Due to Theorem $3.1$, let $\bar{u}$ be the limit of the iterates $\{u_k\}$ of \autoref{(3.2)} with precise data and $\{\delta_m\}$ be a zero sequence. Let us denote the  sequence corresponding to the perturbed data by  $\{v^{\delta_m}\}$ and $k_m=k_*(\delta_m, v^{\delta_m})$ be the stopping rule chosen according to $\autoref{(3.4)}$. Clearly, we have two cases here. First, assume that the sequence $\{k_m\}$ has a finite  accumulation point given by $K$. Without loss of generality, take $k_m = K$ for all $m\in \mathbb{N}$. Then, from \autoref{(3.4)} it follows that $$\|F(w_K^{\delta_m})-v^{\delta_m}\|\leq \tau\delta_m.$$ Further, using  the same formulation via which we deduced  \autoref{E}, we can obtain  $$\|F(u_K^{\delta_m})-F(w_K^{\delta_m})\|\leq\vartheta_7 \|F(w_K^{\delta_m})-v^{\delta_m}\|,$$ for some constant $\vartheta_7$. Combining the last two inequalities yields $$\|F(u_K^{\delta_m})-v^{\delta_m}\|\leq \|F(u_K^{\delta_m})-F(w_K^{\delta_m})\|+\|F(w_K^{\delta_m})-v^{\delta_m}\|$$
$$\hspace{3mm} \leq(1+\vartheta_7)\|F(u_K^{\delta_m})-v^{\delta_m}\|$$ $$\leq (1+\vartheta_7)\tau\delta_m.\hspace{16mm}$$Take $m\to\infty$ which gives $\lim_{m\to\infty}\|F(u_K^{\delta_m})-v^{\delta_m}\|\leq 0$. The continuity of $F$ and $u_K^{\delta_m} \to u_K$ (see Proposition $3.6$)  implies that $F(u_K)=v$. This means that $u_K$ is a solution of \autoref{1.1} in $B(u_0, 2\epsilon)\cap D(\varphi)$. Now as the sequence $\{\mathcal{D}_{\gamma_k}\varphi(u_K, u_k)\}$ is monotonic with respect to $k$, we have $$\mathcal{D}_{\gamma_k}\varphi(u_K, u_k)\leq \mathcal{D}_{\gamma_k}\varphi(u_K, u_K)=0\ \forall\ k\geq K.$$Hence, $u_k=u_K$ for all $k\geq K$. But since $u_k\to \bar{u}$ due to Theorem $3.1$, we have that $u_K=\bar{u}$. By the lower semi-continuity of $\varphi$, note that $$0\leq \liminf_{m\to\infty}\mathcal{D}_{\gamma_{k_m}^{\delta_m}}\varphi(\bar{u}, u_{k_m}^{\delta_m})\hspace{43mm}$$ $$\leq \limsup_{m\to\infty}\mathcal{D}_{\gamma_{k_m}^{\delta_m}}\varphi(\bar{u}, u_{k_m}^{\delta_m})\hspace{38mm} $$ $$\hspace{1mm}\leq\varphi(\bar{u})-\liminf_{m\to\infty}\varphi(u_{k_m}^{\delta_m})-\lim_{m\to\infty}\langle \gamma_{k_m}^{\delta_m}, \bar{u}-u_{k_m}^{\delta_m}\rangle$$ $$\hspace{3mm}\leq \varphi(\bar{u})-\varphi(\bar{u})=0.\hspace{48mm}$$
This means that\begin{equation}\label{M}
 \lim_{m\to\infty}\mathcal{D}_{\gamma_{k_m}^{\delta_m}}\varphi(\bar{u}, u_{k_m}^{\delta_m})=0.\hspace{20mm}\end{equation}This completes the first case. 
  
 Next,  we consider the possibility that the sequence $\{k_m\}$ has no finite accumulation point. Let us fix some  integer $k$, then  $k_m>k$ for some large $m$. Now utilizing Proposition $3.4$ to deduce that $$\mathcal{D}_{\gamma_{k_m}^{\delta_m}}\varphi(\bar{u}, u_{k_m}^{\delta_m})\leq  \mathcal{D}_{\gamma_{k}^{\delta_m}}\varphi(\bar{u}, u_{k}^{\delta_m}) = \varphi(\bar{u})-\varphi(u_{k}^{\delta_m})-\langle \gamma_{k}^{\delta_m}, \bar{u}-u_{k}^{\delta_m}\rangle.$$
At this point, incorporate Proposition $3.6$ to see that $$\gamma_{k}^{\delta_m}\to \gamma_{k}\ \ \text{and}\ \ u_{k}^{\delta_m}\to u_{k}\ \ \text{as}\ \ m\to\infty.$$
With this and the semi-continuity of $\varphi$, observe   that$$0\leq \liminf_{m\to\infty}\mathcal{D}_{\gamma_{k_m}}^{\delta_m}\varphi(\bar{u}, u_{k_m}^{\delta_m})\hspace{42mm}$$ $$\leq \limsup_{m\to\infty}\mathcal{D}_{\gamma_{k_m}}^{\delta_m}\varphi(\bar{u}, u_{k_m}^{\delta_m})\hspace{38mm} $$ $$\leq\varphi(\bar{u})-\liminf_{m\to\infty}\varphi(u_{k}^{\delta_m})-\lim_{m\to\infty}\langle \gamma_{k}^{\delta_m}, \bar{u}-u_{k_m}^{\delta_m}\rangle$$ $$\leq \varphi(\bar{u})-\varphi(u_k)-\langle \gamma_{k}, \bar{u}-u_{k}\rangle=\mathcal{D}_{\gamma_{k}}\varphi(\bar{u}, u_k).$$
Since $k$ can be arbitrary, take limit $k\to \infty$ in above and employ Theorem $3.1$ to deduce that \autoref{M} holds. Hence, the proof is complete.
 \end{proof}
 
\section{Choice of combination parameters $\lambda_k^{\delta}$ and advantages of novel scheme}
In this section, we discuss the various possible choices of  $\lambda_k^{\delta}$ which can be utilized in our scheme {and advantages of our scheme in comparison to \autoref{(1.7)}}. 
Let us recall that $\lambda_k^{\delta}$ must satisfy \autoref{(3.16)} and \autoref{(3.17)} and $\lambda_k$ satisfy \autoref{D}. Obviously, with   $\lambda_k^{\delta}=0$ \autoref{(3.16)} holds and this choice corresponds to iteratively regularized Landweber iteration method \autoref{(1.4)}. However, to have some acceleration for the  scheme \autoref{(3.2)} similar to various Landweber iterations \cite{Hub, Zhong}, non-trivial combination parameters $\lambda_k^{\delta}$ are required.

  To see this, observe that if $\|r_k^{\delta}\|\leq \tau\delta$, then $\upsilon_k^{\delta}=0$ which means $\lambda_k^{\delta}=0$ due to \autoref{(3.16)}. So, the only possibility we need to consider is $\|r_k^{\delta}\|> \tau\delta$. In this situation, \autoref{(3.16)} and \autoref{G} imply the following sufficient condition  on the combination parameters to satisfy \autoref{(3.16)}:\begin{equation}\label{M}
\frac{\lambda_k^{\delta}+(\lambda_{k}^{\delta})^{p^*}}{p^*(2c_0)^{p^*-1}} \|\gamma_{k}^{\delta}-\gamma_{k-1}^{\delta}\|^{p^*}
\leq \frac{\vartheta_5}{\zeta} \min\left\{
\dfrac{(\vartheta_1^{p^*-1}-\bar{\vartheta_2}^{p^*-1})^{\frac{1}{p^*-1}}}{2C_0^p}, \vartheta_3 \right\}(\tau\delta)^{p}.
\end{equation}
Recall that for Nesterov’s acceleration strategy \cite{Nest}, $\lambda_k^{\delta}=\frac{k}{k+\varsigma},$ where $\varsigma\geq 3$. So, by placing the requirement $0\leq \lambda_k^{\delta}\leq \frac{k}{k+\varsigma}<1$ for $p>1$, \autoref{M} leads to the estimate \begin{equation*}
\frac{2\lambda_k^{\delta}}{p^*(2c_0)^{p^*-1}} \|\gamma_{k}^{\delta}-\gamma_{k-1}^{\delta}\|^{p^*}
\leq \frac{\vartheta_5}{\zeta} \min\left\{
\dfrac{(\vartheta_1^{p^*-1}-\bar{\vartheta_2}^{p^*-1})^{\frac{1}{p^*-1}}}{2C_0^p}, \vartheta_3 \right\}(\tau\delta)^{p},
\end{equation*}
 since $(\lambda_{k}^{\delta})^{p^*}<1$. This further means \begin{equation*}
\lambda_k^{\delta} 
\leq \frac{p^*(2c_0)^{p^*-1}\vartheta_5}{2\zeta\|\gamma_{k}^{\delta}-\gamma_{k-1}^{\delta}\|^{p^*}}\min\left\{
\dfrac{(\vartheta_1^{p^*-1}-\bar{\vartheta_2}^{p^*-1})^{\frac{1}{p^*-1}}}{2C_0^p}, \vartheta_3 \right\}(\tau\delta)^{p}.
\end{equation*}
Combining above and $0\leq \lambda_k^{\delta}\leq \frac{k}{k+\varsigma}$, we can select \begin{equation}\label{H}
 \lambda_k^{\delta}= \min\left\{\frac{\vartheta_6\delta^p}{\|\gamma_{k}^{\delta}-\gamma_{k-1}^{\delta}\|^{p^*}},\ \frac{k}{k+\varsigma} \right\}, \ \text{where}\end{equation} $$\vartheta_6=\frac{p^*(2c_0)^{p^*-1}\vartheta_5\tau^{p}}{2\zeta}
\min\left\{
\dfrac{(\vartheta_1^{p^*-1}-\bar{\vartheta_2}^{p^*-1})^{\frac{1}{p^*-1}}}{2C_0^p}, \vartheta_3 \right\}.$$
However, one can see that the choice of combination parameters in \autoref{H} may approach to $0$ whenever $\delta\to 0$. This ultimately implies that the required acceleration effect may also decreases. Therefore, for the sake of completeness and to show the worthness of our posposed scheme, we recall another strategy well known as \emph{discrete backtracking search} (DBTS) algorithm discussed in \cite{Hub, Zhong} to find $\lambda_k^{\delta}$. This strategy fits properly in our situation.

To this end, let us briefly mention some of the used notations which will be used in the following Algorithm $4.1$. $\vartheta_5$ is given by \autoref{(3.14)}, $\zeta$ in \autoref{(3.16)}, $\gamma_k^{\delta}, \gamma_{k-1}^{\delta}$ in \autoref{(3.2)}, $\tau$ in \autoref{(3.4)}. Let $h:\mathbb{N}\cup\{0\}\to (0, \infty)$ be a non increasing function such that $\sum_{k=0}^{\infty}h(k)<\infty.$ This choice of $h$ is needed to satisfy the requirement \autoref{D}.
\vspace{2mm}
\hrule
\raisebox{.5ex}{\rule{2cm}{.4pt}}
\begin{algo}$($DBTS algorithm$)$\vspace{1mm}
\hrule
\raisebox{.5ex}{\rule{12cm}{.1pt}}
\begin{itemize}
\item \textbf{Given} $\gamma_k^{\delta}, \gamma_{k-1}^{\delta},  \tau, \delta, \vartheta_5, h:\mathbb{N}\to \mathbb{N}, \zeta, i_{k-1}^{\delta}\in \mathbb{N}, j_{\max}\in \mathbb{N}$ 
\item \textbf{Set} $\vartheta_7=\frac{p^*(2c_0)^{p^*-1}\vartheta_5}{\zeta}$\ \ \ $($cf. \ $\autoref{(3.16)})$
\item \textbf{Calculate} $\|\gamma_k^{\delta}- \gamma_{k-1}^{\delta}\|$ and define $$\pi_k(i)=\min\bigg\{\frac{h(i)}{\|\gamma_k^{\delta}- \gamma_{k-1}^{\delta}\|}, \ \frac{p^*(2c_0)^{p^*}\epsilon^p}{4\|\gamma_k^{\delta}- \gamma_{k-1}^{\delta}\|^{p^*}}, \ \frac{k}{k+\varsigma}\bigg\},$$
for $\varsigma\geq 3$.
\item \textbf{For} $j=1, 2, \hdots, j_{\max}$

Set $\lambda_k^{\delta}=\pi_k( i_{k-1}^{\delta}+j)$;

Calculate $\Im_k^{\delta}=\gamma_k^{\delta}+\lambda_k^{\delta}(\gamma_k^{\delta}-\gamma_{k-1}^{\delta})$ and $w_k^{\delta}=\nabla\varphi^*(\Im_k^{\delta})$;

Calculate $\upsilon_k^{\delta}$ from $\autoref{(3.3)}$;

\begin{itemize}
\item \textbf{If} $\|F(w_k^{\delta}-v^{\delta}\|\leq \tau \delta$

    $\lambda_k^{\delta}=0$;
    
    $i_k^{\delta}=i_{k-1}^{\delta}+j$;
    
    \textbf{break};
\item \textbf{Else if}   $\autoref{(3.16)}$ holds with $\vartheta_7$ chosen above; 

     $i_k^{\delta}=i_{k-1}^{\delta}+j$;
    
    \textbf{break};
\item \textbf{Else}

     calculate $\lambda_k^{\delta}$ by $\autoref{H}$ 

     $i_k^{\delta}=i_{k-1}^{\delta}+j_{\max}$;
\item \textbf{End if}
\end{itemize}
\item \textbf{End for}
\item \textbf{Output} $\lambda_k^{\delta}$, $i_k^{\delta}$.
\end{itemize}
\end{algo}
\hrule
\raisebox{.5ex}{\rule{2cm}{.4pt}}

The choice of $\pi_k$ in Algorithm $4.1$ will be  cleared shortly in our further discussion. 
From Algorithm $4.1$, it is not hard to see that the combination parameters $\lambda_k^{\delta}$ satisfy \autoref{(3.16)}. This is because the algorithm has the following $3$ possible outputs.\begin{enumerate}
\item $\lambda_k^{\delta}=0$ trivially satisfy
 \autoref{(3.16)}.
\item  $\lambda_k^{\delta}$ satisfy \autoref{(3.16)} in case $\|F(w_k^{\delta}-v^{\delta}\|> \tau \delta$.
\item $\lambda_k^{\delta}$ calculated via \autoref{H} again chosen to satisfy \autoref{(3.16)}. 
\end{enumerate}
However, nothing can be said about the continuous dependence of $\lambda_k^{\delta}$  obtained via Algorithm $4.1$ on $\delta$, whenever $\delta\to 0$. Therefore, Proposition $3.6$ and Theorem $3.2$ are not applicable in this situation. 

Possibly, when $\delta\to 0$, the sequence $\{\lambda_k^{\delta}\}$ may have  many different cluster points. This means in the noise free case, we can have different iterative sequences given by \autoref{(3.2)} by using  different cluster points as the combination parameters. Therefore,  we need to altogether consider these sequences. To this end, let $$\Xi=\{(\Im_k, \gamma_k, u_k, w_k)\in (U^*)^2\times U^2\},$$be the set of sequences defined by \autoref{(2.3)} for $\delta=0$ such that $\{\lambda_k\}$ associated with these sequences  satisfy \autoref{(3.16)} (with $\delta=0$) and \begin{equation}\label{I}
0\leq \lambda_k \leq \pi_k(i_k),\ \ 1\leq i_k-i_{k-1}\leq j_{\max}, \ i_0=0,
\end{equation}
where $\pi_k(i_k)$ are same as in Algorithm $4.1$.

Next, we show that with these choice of parameters $\{\lambda_k\}$, Theorem $3.1$ is applicable. Note  that \autoref{(3.16)} holds trivially.
 Due to \autoref{I}, we have $$0\leq \lambda_k\leq \frac{p^*(2c_0)^{p^*}\epsilon^p}{4\|\gamma_k- \gamma_{k-1}\|^{p^*}}\ \ \text{and}\ \ 0\leq \lambda_k\leq\frac{k}{k+\varsigma}<1.$$ This with the left side of \autoref{(3.17)} implies that$$ \frac{\lambda_k+(\lambda_{k})^{p^*}}{p^*(2c_0)^{p^*-1}} \|\gamma_{k}-\gamma_{k-1}\|^{p^*}\leq \frac{2\lambda_k}{p^*(2c_0)^{p^*-1}} \|\gamma_{k}-\gamma_{k-1}\|^{p^*} \leq c_0 \epsilon^p.$$Thus, \autoref{(3.17)} holds. Now, we talk about \autoref{D}. To see this, from \autoref{I}, we have $\lambda_k\leq \frac{h(i_k)}{\|\gamma_k- \gamma_{k-1}\|}$. Consider  the left hand side of \autoref{D} to obtain $$\sum_{k=0}^{\infty}\lambda_k\|\gamma_k-\gamma_{k-1}\|\leq \sum_{k=0}^{\infty}h(i_k).$$ Since $h$ is a non-increasing function and  $i_k\geq k$, we have $h(i_k)\leq h(k)$. Thus, \autoref{D} is satisfied provided $\sum_{k=0}^{\infty}h(k)<\infty$ which already holds for the function $h$. So, Theorem $3.1$ is applicable for these choices of parameters.

Finally, we discuss the two results related to the incorporation of DBTS algorithm in our scheme. A stability  result similar to Proposition $3.6$ without requiring the continuous dependence of $\lambda_k^{\delta}$ on $\delta$   also holds for the combination parameters chosen via DBTS algorithm. Its proof can be developed on the similar lines of \cite[Lemma $3.9$]{Zhong}.
Using this stability result, one can show the regularizing nature of the method \autoref{(3.2)}, 
whenever combination parameters are chosen via DBTS algorithm. We again skip its proof as it can be  developed on the similar lines of \cite[Theorem $3.10$]{Zhong}. However, for the sake of completeness, we give both the results in the following proposition and theorem.
\begin{prop}
Let $V$ be uniformly smooth, $U$ be reflexive and assumptions of Subsection $3.1$ be satisfied.
Let the noisy data $\{v^{\delta_m}\}$ be such that $\autoref{1.2}$ holds and $\delta_m\to 0$ as $m\to\infty$. Further, assume that $\vartheta_5$ in $\autoref{(3.14)}$ is positive and the combination parameters $\lambda_k^{\delta_m}$ are deduced via Algorithm $4.1$ and $\lambda_0^{\delta_m}=0$. Then by taking a subsequence of $\{v^{\delta_m}\}$, there exists a sequence $\{(\Im_k, \gamma_k, u_k, w_k)\}\in \Xi_m$ such that  for all $k\geq 0$ we have$$\Im_k^{\delta_m}\to \Im_k, \ \ \ \gamma_k^{\delta_m}\to\gamma_k,\ \ \ u_k^{\delta_m}\to u_k, \ \ \ w_k^{\delta_m}\to w_k\ \ \ \text{as}\ \ \  m\to \infty.$$ 
\end{prop}
\begin{thm}
Let $V$ be uniformly smooth, $U$ be reflexive and assumptions of Subsection $3.1$ be satisfied.
Let the noisy data $\{v^{\delta}\}$ be such that $\autoref{1.2}$ holds. Further, assume that $\vartheta_5$ in $\autoref{(3.14)}$ is positive and the combination parameters $\lambda_k^{\delta_m}$ are deduced via Algorithm $4.1$ Let $k_{\delta}$ be the integer determined through $\autoref{(3.4)}$. Then for any subsequence $\{v^{\delta_m}\}$ of $\{v^{\delta}\}$ with $\delta_m\to 0$ as $m\to\infty$, by taking a subsequence of $\{v^{\delta_m}\}$, if necessary, there hold$$\lim_{m\to\infty}\mathcal{D}_{\gamma_{k_{\delta_m}}}^{\delta_m}\varphi(\bar{u}, u_{k_{\delta_m}}^{\delta_m})=0,\ \text{and}\ \ \lim_{m\to\infty}\|u_{k_{\delta_m}}^{\delta_m}-\bar{u}\|=0.$$
\end{thm}
{Finally, we end this subsection with the following remark in which we compare our method with  the method proposed in \cite{Zhong}}.

\begin{rema}
{ The two-point gradient method  proposed by  Zhong et al. \cite{Zhong} in Banach spaces  is based on the Landweber iteration and an extrapolation strategy. Motivated by this method, in this paper we proposed a novel two-point gradient method  based on  the modified Landweber iteration together with an  extrapolation strategy. To be more precise, on taking $\alpha_k=0\ \forall k$ in \autoref{(3.2)}, our method reduces to the method \autoref{(1.7)} stated in \cite{Zhong}. Since the iteratively regularized Landweber iteration is a generalization of the Landweber iteration (cf \cite{Scherzer} to see the advantages of iteratively regularized Landweber iteration over Landweber iteration), the main advantage of our novel two point gradient method \autoref{(3.2)} over the method  \autoref{(1.7)} of \cite{Zhong} is that   \autoref{(3.2)} is a generalized version of \autoref{(1.7)}}.

{Unfortunately, we don't have our own numerical results but we emphasize the fact that for $\alpha_k=0$ in \autoref{(3.2)}, all the numerical simulations presented in \cite{Zhong} also hold for  our scheme. This follows from \autoref{(3.2)} and proof of Proposition $3.2$, as  one can see that if $\alpha_k=0$ for all $k$, then the point $(4)$ of assumptions in subsection $3.1$ is not required in our analysis. And accordingly we need to update  $\upsilon_k^{\delta}$ in \autoref{(3.3)} and several other constants. Consequently, for $\alpha_k=0$ in \autoref{(3.2)}, our assumptions become exactly similar to that of \cite{Zhong} and, therefore,  this paper extends  the results of \cite{Zhong} for a generalized version of \autoref{(1.7)}}.
\end{rema}

\section{Example: Electrical Impedance Tomography (EIT)}
The main aim of this section is to show the validity of our method for which we discuss an example that satisfy the assumptions required in our framework. We consider  a  severely ill-posed Calder\'on's inverse problem which is the mathematical bedrock of  EIT \cite{Aless}. 

Let $\Psi\subset \mathbb{R}^n, n\geq 2$ be a bounded domain having smooth boundary and $u \in H^1(\Psi)$ satisfies the following Dirichlet problem:
 \begin{equation}\label{(5.1)}
 \begin{cases} \text{div}(\kappa \nabla u) = 0, \ \  \text{in} \  \Psi \\ \ \   u = f, \qquad   \ \quad \text{on} \ \partial \Psi.\end{cases} \end{equation} Here $f \in H^{1/2}(\partial \Psi)$ and $\kappa$ is the positive and bounded function representing the electrical conductivity of $\Psi$. Calder\'on's inverse  problem has many applications, for instance, in the fields of medical imaging,  nondestructive testing of materials etc. (cf. \cite{Aless}). 

The inverse problem associated with 
EIT can be formulated  as \begin{equation}\label{(5.2)}F : U \subset L_{+}^{\infty}(\Psi) \to L(H^{1/2}(\partial \Psi), H^{-1/2}(\partial \Psi)): F(\kappa) = \Lambda_{\kappa},\end{equation} where $L(H^{1/2}(\partial \Psi), H^{-1/2}(\partial \Psi))$ is the space of all bounded linear operators from $H^{1/2}(\partial \Psi)$ to  $H^{-1/2}(\partial \Psi)$ and Dirichlet to Neumann map $\Lambda_{\kappa}$ is defined as $$\Lambda_{\kappa}: H^{1/2}(\partial \Psi) \to H^{-1/2}(\partial \Psi): \ f \to \bigg(\kappa \frac{\partial u}{\partial \nu}\bigg)\bigg|_{\partial \Psi},$$ where the vector $\nu$ is the outward normal to $\partial \Psi$. The Fr\'echet derivative $F'$ of $F$  at $\kappa = \bar{\kappa}$ is given by \begin{equation*}
 F'(\bar{\kappa}): U \subset L^{\infty}(\Psi) \to L(H^{1/2}(\partial \Psi), H^{-1/2}(\partial \Psi)): \delta \kappa \to \ F'(\bar{\kappa})(\delta \kappa), 
 \end{equation*} where $F'(\bar{\kappa})(\delta \kappa)$ is defined via the following sesquilinear form \begin{equation*}
 \langle F'(\bar{\kappa})(\delta \kappa) f_1, \ f_2  \rangle = \int_{\Psi} \delta\kappa \nabla u_1 \cdot \nabla u_2\, dx, \quad f_1, f_2 \in H^{1/2}(\partial \Psi),\end{equation*} where $u_1$ and $u_2$ are the weak solutions of $$\begin{cases}
 \text{div} (\bar{\kappa} \nabla u_1) = 0 = \text{div}  (\bar{\kappa} \nabla u_2), \quad \text{in} \  \Psi \\ u_1 = f_1, \ \ u_2 = f_2 \qquad \quad \quad \qquad\ \  \text{on} \ \partial \Psi. \end{cases}$$ 
 For $n=2$ and the condition $\kappa \in L^{\infty}(\Psi)$, uniqueness of the solution to the inverse problem \autoref{(5.2)} has been discussed in \cite{Astala}. For $n\geq 3$, uniqueness of \autoref{(5.2)} has been discussed in \cite{Paiv} under the assumption that $\kappa\in W^{3/2, \infty}(\Psi)$. 
 
 To this end, let us recall a  Lipschitz estimate  established in \cite{Aless} for the inverse problem \autoref{(5.2)} under certain assumptions.
\begin{thm}
Let $\kappa_1, \kappa_2$ be two real piecewise constant  functions  such that $$\kappa_i(x)=\sum_{j=1}^N \kappa_j^i(x)\chi_{D_j}(x),\ x\in \Psi,\ \lambda\leq \kappa_i(x)\leq \lambda,\ i=1, 2,$$ where $\lambda\in (0, 1]$, $\kappa_j^i$ is an   unknown real number for each $i, j$, $\chi_{D_j}$ is a  characteristics function of the set $D_j$, $D_j$'s are known open sets,   and $N\in \mathbb{N}$. Then under certain assumptions on $\Psi$, $D_j$'s $($cf. \cite[Section $2.2$]{Aless}$)$,  we have \begin{equation*}
\|\kappa_1-\kappa_2\|_{L^{\infty}(\Psi)} \leq C \|\Lambda_{\kappa_1}-\Lambda_{\kappa_2}\|_{L(H^{1/2}(\partial \Psi), H^{-1/2}(\partial \Psi))},
\end{equation*} where $C$ is a constant. 
\end{thm}
Next, let us define the space $U$ in accodance with  Theorem $5.1$ as\begin{equation*}
 U= \text{span} \{\chi_{D_1}, \chi_{D_2}, \ldots, \chi_{D_N}\}\end{equation*} fitted with $L^p$ norm where $p > 1$, $D_i$'s, $\chi_{D_i}'s$ are same as in Theorem $5.1$. This choice of $U$ renders it as a reflexive Banach space. Additionally, we know the following result related to the continuity of $F$ and $F'$ and boundedness of $F'$ (cf. \cite[Subsection\ 5.3]{Hoop1}).
 
\begin{thm}Let the assumptions of Theorem $5.1$ hold. Then, we have the following:\begin{enumerate}
\item $F$ is Lipschitz continuous and satisfies the estimate \begin{equation*}
\|F(\kappa_1)-F(\kappa_2)\|_{L(H^{1/2}(\Psi), H^{-1/2}(\Psi))} \leq C_1 \|\kappa_1-\kappa_2\|_{L^p(\Psi)},
\end{equation*}
\item $F'$ is Lipschitz continuous and satisfies the estimate \begin{equation*}
\|F'(\kappa_1)-F'(\kappa_2)\|_{L(H^{1/2}(\Psi), H^{-1/2}(\Psi))} \leq C_2 \|\kappa_1-\kappa_2\|_{L^p(\Psi)},
\end{equation*}
\item $F'$ is bounded, i.e. $$\|F'\|_{L(U,L(H^{1/2}(\Psi), H^{-1/2}(\Psi)))}\leq C_3,$$ 
\end{enumerate}
where $C_1, C_2$ and $C_3$ are constants.
\end{thm}
 Finally, let us discuss how the assumptions of Subsection $3.1$ are satisfied for this inverse problem for $\varphi(u)=\frac{\|u\|_{L^{\infty}(\Psi)}^2}{2}$.
 \begin{enumerate}
\item Point $(2)$ of assumptions in Subsection $3.1$ holds by considering $\varphi(u)=\frac{\|u\|_{L^{\infty}(\Psi)}^2}{2}$.
 
 \item Point $(4)$ of assumptions Subsection $3.1$ holds for $p=2$ and $\varphi(u)=\frac{\|u\|_{L^{\infty}(\Psi)}^2}{2}$ via  Theorem $5.1$.
 \item Since the notions of strong topology and weak topology are same for finite dimensional space $U$, Point $(1)$ of assumptions in Subsection $3.1$ holds.
 \item   By considering $L(u)=F'(u)$, Point $(6)$ of assumptions in Subsection $3.1$ holds due to  $(3)$ of Theorem $5.2$.
 \item Consider the left side of inequality in Point $(5)$ of assumptions of Subsection $3.2$:$$\|F(\kappa_1) - F(\kappa_2) - F'(\kappa_2)(\kappa_1-\kappa_2)\|_{L(H^{1/2}(\Psi), H^{-1/2}(\Psi))}\hspace{40mm}$$ $$\hspace{5mm} \leq \|F(\kappa_1) - F(\kappa_2)\|_{L(H^{1/2}(\Psi), H^{-1/2}(\Psi))}+\|F'(\kappa_2)(\kappa_1-\kappa_2)\|_{L(H^{1/2}(\Psi), H^{-1/2}(\Psi))}.$$Utilize points $(1)$ and $(3)$ of Theorem $5.2$ in above to obtain $$\|F(\kappa_1) - F(\kappa_2) - F'(\kappa_2)(\kappa_1-\kappa_2)\|_{L(H^{1/2}(\Psi), H^{-1/2}(\Psi))}\hspace{40mm}$$ \begin{equation}\label{(5.3)}
   \leq C_1\|\kappa_1 - \kappa_2\|_{L^{p}(\Psi))}+C_3\|\kappa_1-\kappa_2\|_{L^{p}(\Psi)}.\end{equation}
We also know that $L^p$ norm is bounded by the $L^{\infty}$ norm, i.e. $$\|f\|_{L^{p}(\Psi)}\leq \mu(\Psi)^{\frac{1}{p}}\|f\|_{L^{\infty}(\Psi)},$$ where $\mu(\Psi)$ is the measure of $\Psi$. Substituting above inequality in $(5.3)$ to deduce that $$\|F(\kappa_1) - F(\kappa_2) - F'(\kappa_2)(\kappa_1-\kappa_2)\|_{L(H^{1/2}(\Psi), H^{-1/2}(\Psi))}\hspace{40mm}$$ $$\leq  \mu(\Psi)^{\frac{1}{p}}(C_1+C_3)\|\kappa_1-\kappa_2\|_{L^{\infty}(\Psi)}.$$
Thus, the required tangential cone condition holds,  provided $\mu(\Psi)^{\frac{1}{p}}(C_1+C_3)<1$ for $p=2$.
\end{enumerate}  
Since all the assumptions of Subsection $3.2$ hold, we conclude that our novel algorithm \autoref{(3.2)}  can be applied to solve severely ill-posed EIT problem.

\section{Discussion}
We have shown the convergence of a novel two point gradient method  \autoref{(3.2)} obtained by combining iteratively regularized Landweber iteration method together with an extrapolation strategy under classical assumptions. We have also discussed various possibilities for combination parameters together with an algorithm known as DBTS algorithm considerd in \cite{Hub}. Although no numerical results are yet available but we have discussed an example of severe ill-posed problem   which satisfy our assumptions.  Moreover, the theory is developed such that it remains compatible with the two point gradient method  introduced in \cite{Zhong} since our method with $\alpha_k=0$ in \autoref{(3.2)} reduces to the method discussed in \cite{Zhong}. 

Due to the numerical as well as analytic demonstration of the great reduction of the required number of iterations in \cite{Hub, Zhong}, two point gradient   methods can be a really good replacement to well known ‘fast’ iterative methods, like newton type methods, especially when large-scale inverse problems are considered. This is because one need to solve gigantic linear systems in each iteration step for  the  fast known iterative methods and therefore   they often become impracticable.

\bibliographystyle{plain}

\begin{thebibliography}{10}
\small{
 \bibitem{Aless}{G. Alessandrini and  S. Vessella}, \newblock{\em Lipschitz stability for the inverse conductivity problem},\newblock{ Adv. in Appl. Math.,  $35(2)$,    $207-241$, $2005$.}
 
 \bibitem{Astala}{K. Astala and  L. P\"aiv\"arinta,} \newblock{\em Calder\'on's inverse conductivity problem in the plane},\newblock{ Ann. of Math.,  $163$,    $265-299$, $2006$.}

\bibitem{Bot}{ R. Bot  and T. Hein,} \newblock{\em Iterative regularization with a general penalty term: theory and applications to $L^1$ and
TV regularization},\newblock{ Inverse Probl., 28, 104010, 2012.}


\bibitem{Engl}{ H.W. Engl, M. Hanke and A. Neubauer,} \newblock{\em Regularization of inverse problems},\newblock{ Springer Netherlands, $2000$.}

\bibitem{Hanke2}{ M. Hanke}, \newblock{\em Accelerated Landweber iterations for the solution of ill-posed equations},\newblock{  Numer. Math., 60, 341-373, 1991.}

\bibitem{Hanke}{ M. Hanke,  A. Neubauer and O. Scherzer}, \newblock{\em A convergence analysis of the Landweber iteration for nonlinear ill-posed problems},\newblock{  Numer. Math., 72, 21-37, 1995.}




\bibitem{Hegland}{ M. Hegland, Q. Jin  and W.  Wang}, \newblock{\em Accelerated Landweber iteration with convex penalty for linear inverse
problems in Banach spaces},\newblock{ Appl. Anal., 94, 524-547, 2015.}

\bibitem{Hein}{ T. Hein and K.S. Kazimierski}, \newblock{\em Accelerated Landweber iteration in Banach spaces},\newblock{  Inverse Probl., 26, 1037-1050, 2010.}


\bibitem {Hoop1}{  M.V. de Hoop, L. Qiu and  O. Scherzer,} \newblock{\em Local analysis of  inverse problems: H\"older stability and iterative reconstuction,}  \newblock{Inverse Probl., $28(4)$, $045001,$ pp. $16$,  $2012$.}

\bibitem {Hoop2}{  M.V. de Hoop, L. Qiu and  O. Scherzer,} \newblock{\em An analysis of a multi-level projected steepest descent iteration for nonlinear inverse problems in Banach spaces
subject to stability constraints,}  \newblock{Inverse Probl., $129$,   127-148,  $2015$.}

\bibitem{Hub}{ S. Hubmer  and R.  Ramlau}, \newblock{\em Convergence analysis of a two-point gradient method for nonlinear ill-posed
problems},\newblock{ Inverse Probl., 33, 095004, 2017.}

\bibitem{Jin1}{ Q. Jin}, \newblock{\em Inexact Newton-Landweber iteration for solving nonlinear inverse problems in Banach spaces},\newblock{  Inverse Probl., 28, 065002, 2012.}


\bibitem{Jin2}{ Q. Jin and W. Wang}, \newblock{\em Landweber iteration of Kaczmarz type with general non-smooth convex penalty
functionals},\newblock{  Inverse Probl., 29, 085011, 2013.}


\bibitem{Jin3}{ Q. Jin}, \newblock{\em Landweber–Kaczmarz method in Banach spaces with inexact inner solvers},\newblock{  Inverse Probl., 32, 104005, 2016.}


\bibitem{Kalten}{ B. Kaltenbacher,  A. Neubauer and O. Scherzer,} \newblock{\em Iterative regularization methods for nonlinear Ill-posed problems},\newblock{ De Gruyter, $2008$.}

\bibitem{Nest}{ Y. Nesterov}, \newblock{\em A method of solving a convex programming problem with convergence rate $O(1/k^2)$},\newblock{ Sov. Math. Dokl., 27,  372-376, 1983.}

\bibitem{Paiv}{L. P\"aiv\"arinta, A. Panchenko and  G. Uhlmann}, \newblock{\em Complex geometrical optics solutions for Lipschitz conductivities},\newblock{ Rev. Mat. Iberoam.,  $19(1)$,    57-72, $2003$.}

\bibitem{Scherzer}{ O. Scherzer}, \newblock{\em A modified Landweber iteration for solving parameter estimation problems},\newblock{  Appl. Math. Optim., 38, 45-68 , 1998.}

\bibitem{Shop}{ F. Sch\"opfer, A.K. Louis  and T.  Schuster}, \newblock{\em Fast regularizing sequential subspace optimization in Banach spaces},\newblock{ Inverse Probl., 25, 015013, 2009.}

\bibitem{Schuster}{ T. Schuster, B. Kaltenbacher, B. Hofmann and  K. S. Kazimierski,} \newblock{\em Regularization methods in Banach spaces},\newblock{ De Gruyter,  $2012$.}



\bibitem{Wang}{ J. Wang, W. Wang and  K. S. Kazimierski,} \newblock{\em An iteration regularizaion method with general convex penalty for nonlinear inverse problems in Banach spaces},\newblock{ J. Comput. Appl. Math.,  361, 472-486, 2019.}



\bibitem{Zhong}{M. Zhong, W. Wang and  Q. Jin,}  \newblock{\em Regularization of inverse problems by two-point gradient methods in Banach spaces},\newblock{  Numer. Math., 143, 713-747, 2019.}

\bibitem{Z\u alinscu}{C. Z\u alinscu,}  \newblock{\em Convex Analysis in General Vector Spaces},\newblock{  World Scientific Publishing Co., Inc., River
Edge,  2002.}

 
\vspace{3mm}
\small}

 \end{thebibliography}

\end{document}